\documentclass{article}

\usepackage{amssymb}
\usepackage{latexsym}
\usepackage{amsthm}
\usepackage{enumerate}
\usepackage{epsfig}
\usepackage{color}
\usepackage{float}
\usepackage{subfigure}
\usepackage{amsmath}
\usepackage{ifthen}
\usepackage{tikz}
\usetikzlibrary{arrows}
\usepackage[nobysame, alphabetic]{amsrefs}
\usepackage{wrapfig}
\usepackage{graphicx}
\usepackage{sidecap}

\def\calB{\mathcal{B}}
\def\bdy{\partial}
\def\C{\mathcal{C}}
\def\calC{\mathcal{C}}
\def\calD{\mathcal{D}}

\def\calH{\mathcal{H}}

\def\calM{\mathcal{M}}
\def\N{\mathbb{N}}

\def\R{\mathbb{R}}

\def\Z{\mathbb{Z}}

\def\cross{\times}

\def\P{\mathbb{P}}
\def\calY{\mathcal{Y}}
\def\calZ{\mathcal{Z}}

\def\PML{\mathcal{PML}}

\def\M{\mathcal{M}}

\def\calF{\mathcal{F}}

\def\bdy{\partial}
\def\from{\colon}

\def\CB{\mathcal{H}}

\def\carried{\prec}
\def\carries{\succ}

\renewcommand{\ge}{\geqslant}
\renewcommand{\le}{\leqslant}

\renewcommand{\setminus}{-}

\newcommand{\norm}[1]{|#1|}

\newcommand{\dg}[1]{d_{\text{Mod}}(#1)}

\newcommand{\gp}[3]{(#2 \cdot #3)_{#1}}

\newtheorem{theorem}{Theorem}[section]
\newtheorem{lemma}[theorem]{Lemma}
\newtheorem{corollary}[theorem]{Corollary}
\newtheorem{proposition}[theorem]{Proposition}

\theoremstyle{definition}

\newtheorem{definition}[theorem]{Definition}

\restylefloat{figure}

\begin{document}


\title{The compression body graph has infinite diameter}
\author{Joseph Maher, Saul Schleimer}
\date{\today}

\maketitle

\begin{abstract}
We show that the compression body graph 
has infinite diameter. \\


Subject code: 37E30, 20F65, 57M50.
\end{abstract}



\tableofcontents

\section{Introduction}

The curve complex is a simplicial complex whose vertices are isotopy
classes of simple closed curves, and whose simplices are spanned by
simple closed curves which may be realized disjointly in the surface.
In this paper, we consider a related complex, the \emph{compression
  body graph}, defined by Biringer and Vlamis \cite{bv}*{page~94},
which we shall denote by $\mathcal{H}(S)$.  Vertices of this graph are
isomorphism classes of \emph{marked compression bodies}.  A marked
compression body $(U, f)$ is a non-trivial compression body $U$
together with a choice of homeomorphism $f \colon \partial_+ U \to S$,
up to isotopy, from the upper boundary of the compression body to the
surface $S$.  Two marked compression bodies $(U, f)$ and $(V, g)$ are
\emph{isomomorphic} if they differ by a homeomorphism of the
compression body.  More precisely, there must be a homeomorphism
$h \colon U \to V$ such that $f = g \circ h \vert_{\partial_+ U}$.
Two compression bodies are \emph{adjacent} if one is contained in the
other one.  More precisely, the compression body graph is a poset as
follows: suppose that $(U, f)$ and $(V, g)$ are marked compression
bodies, and $ U' \subseteq U $ is a subcompression body of $U$.  If
there is a homeomorphism $h \colon U' \to V$ such that
$f = g \circ h \vert_{\partial_+ U}$ then $V < U$.  Biringer and
Vlamis \cite{bv}*{Theorem~1.1}, following Ivanov, showed that the
simplicial automorphism group of this graph is equal to the mapping
class group. We show:

\begin{theorem}
\label{Thm:InfDiam}
The compression body graph $\mathcal{H}(S)$ is an infinite diameter
Gromov hyperbolic metric space.
\end{theorem}

The compression body graph $\mathcal{H}(S)$ is quasi-isometric to the
metric space obtained by electrifying the curve complex along each
\emph{disc set}: the set of all simple closed curves that bound discs
in a specific compression body.
%
%
We will write $\pi_\calY$ for the inclusion map
$X \hookrightarrow X_\calY$, which we shall also refer to as the
\emph{projection map}.  As the electrification of a Gromov hyperbolic
space along uniformly quasi-convex subsets is Gromov hyperbolic the
compression body graph is Gromov hyperbolic.  This follows from work
of Bowditch \cite{bowditch-rel-hyp}, Kapovich and Rafi
\cite{kapovich-rafi} and Masur and Minsky \cite{mm1, mm3}.  See
Section \ref{section:outline} for further details.




We also study the action of the mapping class group on $\calH(S)$.  In
Lemma \ref{lemma:coarse} we show that that for every genus $g \ge 2$,
there are elements of the mapping class group of $S$ which act
loxodromically on $\calH(S)$.  Furthermore in Corollary
\ref{corollary:johnson} we show that every subgroup of the Johnson
filtration contains elements which act loxodromically on $\calH(S)$.

In the final section, we use our methods to give an alternate proof of
Theorem \ref{T:bjm}: the stable lamination of a pseudo-Anosov element
is contained in the limit set of a compression body $V$ if and only if
some power of the pseudo-Anosov extends over a non-trivial
subcompression body of $V$.  This was originally shown by Biringer,
Johnson and Minsky \cite{bjm}*{Theorem 1}, using techniques from
hyperbolic three-manifolds.  It has also been shown by Ackermann
\cite{ack}*{Theorem 1}, extending the methods of Casson and Long
\cite{casson-long}.  We also weaken the hypotheses of a result of
Lubotzky, Maher and Wu \cite{lmw}*{Theorem 1}, showing that a random
Heegaard splitting is hyperbolic with probability tending to one
exponentially quickly, for a larger class of random walks than those
considered in \cite{lmw}.

The results of this paper have been used by Agol and
Freedman~\cite{agol-freedman}, Burton and
Purcell~\cite{burton-purcell}*{Theorem 4.12}, Dang and
Purcell~\cite{dang-purcell}*{Theorem 1.2} and Ma and
Wang~\cite{ma-wang}.

Finally, Theorem \ref{Thm:InfDiam} implies that many graphs, formed by
considering compression bodies of restricted topological type, are
also of infinite diameter, see Section \ref{section:restricted} for
further details.

\subsection*{Acknowledgements} \label{section:Acknowledgments}

We thank Ian Biringer for enlightening conversations.  The first
author acknowledges support from the Simons Foundation and PSC-CUNY.
The second author acknowledges support from the EPSRC (EP/I028870/1).
The authors further acknowledge the support of the NSF (DMS-1440140)
while the both were in residence at the MSRI in Berkeley, California,
during the Fall 2016 semester.  The authors would like to thank the
referee for helpful comments and suggestions.

\section{Outline} \label{section:outline}

We now give a brief outline of the main argument.  The curve complex
is Gromov hyperbolic, and the disc sets are quasi-convex.  Thus, by
electrifying the curve complex along the discs sets, we obtain a
Gromov hyperbolic space $\C_\calD(S)$ which is quasi-isometric to the
compression body graph $\mathcal{H}(S)$.
The electrification of a Gromov hyperbolic
space along uniformly quasi-convex subsets is Gromov hyperbolic:

\begin{theorem} \label{theorem:quotient} Let $X$ be a Gromov
hyperbolic space and let $\calY = \{ Y_i\}_{i \in I}$ be a collection
of uniformly quasi-convex subsets of $X$.  Then $X_\calY$, the
electrification of $X$ with respect to $\calY$, is Gromov hyperbolic.

Furthermore, for any constants $K$ and $c$, there are constants $K'$
and $c'$ such that a $(K, c)$-quasi-geodesic in $X$ projects to a
reparameterized $(K', c')$-quasi-geodesic in $X_\calY$.
\end{theorem}

The first statement is due to Bowditch
\cite{bowditch-rel-hyp}*{Proposition 7.12}, and both statements follow
immediately from Kapovich and Rafi \cite{kapovich-rafi}*{Corollary
  2.4}.  Masur and Minsky showed that the curve complex is Gromov
hyperbolic \cite{mm1}*{Theorem 1.1}, and that the disc sets are
quasi-convex \cite{mm3}*{Theorem~1.1}.  We shall write $\C(S)$ for the
curve graph of the surface $S$, and $\calD$ for the collection of all
discs sets in $\C(S)$. Therefore $\C_\calD(S)$ denotes the curve
complex electrified along disc sets, which is quasi-isometric to the
compression body graph $\calH(S)$.  These two quasi-isometric spaces
are Gromov hyperbolic with infinite diameter by Theorem
\ref{Thm:InfDiam}.  Furthermore, work of Dowdall and Taylor
\cite{dow-tay}*{Theorem 3.2} shows that the Gromov boundary of
$\C_\calD(S)$ is homeomorphic to the subset of $\partial \C(S)$
consisting of the complement of the limit sets of disc sets.

Any quasi-geodesic ray $\gamma \colon \N \to \C(S)$ projects to a
reparameterized and possibly finite diameter quasi-geodesic ray in the
electrification $\C_\calD(S)$.  In Section \ref{section:stability}, we
prove our stability result for quasi-geodesics in the curve complex
$\C(S)$: if $\pi_\mathcal{H} \circ \gamma$ has finite diameter image
in $\C_\calD(S)$ then there is a constant $k$, and a handlebody $V$,
such that $\gamma$ is contained in a $k$-neighbourhood of the disc set
$\mathcal{D}(V)$ in $\C(S)$. In particular, this means that the ending
lamination corresponding to $\gamma$ lies in the limit set of the disc
set $\mathcal{D}(V)$.  However, for a given handlebody $V$, the limit
set of its disc set $\mathcal{D}(V)$ in the boundary of the curve
complex has measure zero.  As there are only countably many discs
sets, the union of their limit sets also has measure zero.  Therefore,
there is a full measure set of minimal laminations disjoint from the
union of the disc sets.  Any one of these gives rise to a geodesic ray
whose image in the electrification $\C_\calD(S)$ has infinite
diameter.  In particular, $\C_\calD(S)$ has infinite diameter.

We now give a brief sketch of the proof of the stability result.
Suppose there is a geodesic ray $\gamma$ in the curve complex $\C(S)$,
and a sequence of handlebodies $V_i$, such that the initial segment of
$\gamma$ of length $i$ is contained in a $k'$-neighbourhood of
$\mathcal{D}(V_i)$.  In particular, for any $i$, there are infinitely
many disc sets $\mathcal{D}(V_j)$ passing within distance $k'$ of both
$\gamma_0$ and $\gamma_i$.

Recall that given two simple closed curves $a$ and $b$ on a surface
$S$, we may \emph{surger} $a$, along an innermost arc of $b$, to
produce a new curve $a'$, disjoint from $a$, and with smaller
geometric intersection number with $b$.  This is illustrated in Figure
\ref{fig:arc}.  By iterating this procedure, we obtain a surgery
sequence, which gives rise to a reparameterized quasi-geodesic in
$\C(S)$ from $a$ to $b$.  We call this a \emph{curve surgery
  sequence}.  If the simple closed curves $a$ and $b$ bound discs in a
handlebody $V$, then we may surger $a$ along innermost bigons of $b$,
to produce a surgery sequence connecting $a$ and $b$, in which every
surgery curve bounds a disc in $V$.  We shall call such a surgery
sequence a \emph{disc surgery sequence}.

Recall that a \emph{train track} on a surface $S$ is a smoothly
embedded graph, such that the edges at each vertex are all mutually
tangent, and there is at least one edge in each of the two possible
directed tangent directions.  Furthermore, there are no complementary
regions which are nullgons, monogons, bigons or annuli.

A \emph{split} of a train track $\tau$ is a new train track $\tau'$
obtained by one of the local modifications illustrated in Figure
\ref{pic:split}.  We say that a sequence of train tracks
$\{ \tau_j \}$ is a \emph{splitting sequence} if each $\tau_{j+1}$ is
obtained as a split of $\tau_j$.

A key result we use from \cite{mm3}*{page 319} says, roughly speaking,
that we may choose a surgery sequence $\{ D_j \}$ connecting two discs
in a common compression body such that there is a corresponding train
track splitting sequence $\{ \tau_j \}$, so that for each $j$, the
disc $D_j$ is dual to the train track $\tau_j$, and meets it in a
single point at the switch.  We state a precise version of this as
Proposition \ref{prop:splitting seq} below.

Recall that given an essential subsurface $X$ contained in $S$, there
is a \emph{subsurface projection map} from a subset of $\C(S)$ to
$\C(X)$.  Roughly speaking, this map sends a simple closed curve $a$,
meeting $X$ essentially, to a simple closed curve $a' \subset X$,
which is disjoint from some arc of $a \cap X$.  See Section
\ref{section:prelim} for a precise definition.  We may extend the
definition of subsurface projection from simple closed curves to train
tracks, by instead projecting the \emph{vertex cycles} of the train
track $\tau$.  See Section \ref{section:tt} for further details.  The
collection of vertex cycles $\{ \Lambda(\tau_j) \}$ for a splitting
sequence $\{ \tau_j \}$ projects to a reparameterized quasi-geodesic
in $\C(S)$.

We say that three (ordered) points $x$, $y$ and $z$ in a metric space
satisfy the \emph{reverse triangle inequality} with constant $K$ if
$d(x, y) + d(y, z) \le d(x, z) + K$.  By the Morse lemma, given
constants $\delta$, $Q$ and $c$, there is a constant $K$, such that if
$y$ lies on a $(Q, c)$-quasi-geodesic between $x$ and $y$ in a
$\delta$-hyperbolic space, then $x$, $y$ and $z$ satisfy the reverse
triangle inequality.  Given three (ordered) points $x$, $y$ and $z$,
we say that $y$ is \emph{$K$-intermediate} with respect to $x$ and
$z$, if $x$, $y$ and $z$ satisfy the reverse triangle inequality, and
furthermore $d(x, y) \ge K$ and $d(y, z) \ge K$.

Given a train track sequence $\{ \tau_j \}$, we say that a particular
train track $\tau_j$ is \emph{$K$-intermediate} if the vertex cycles
of $\tau_j$ are distance at least $K$ in the curve complex from the
vertex cycles of both $\tau_0$ and $\tau_n$.  If $K$ is sufficiently
large, then for any two train track sequences $\{ \tau_j \}$ and
$\{ \tau'_j \}$ starting near $\gamma_0$ and ending near $\gamma_r$,
for any subsurface $X$ in $S$, and for any pair of $K$-intermediate
train tracks $\tau_j$ and $\tau'_k$, the distance between the
subsurface projections of $\tau_j$ and $\tau'_k$ in $\mathcal{C}(X)$
is bounded in terms of $d_X(\gamma_0, \gamma_r)$.  In particular,
using the Masur-Minsky distance formula we find that every train track
sequence connecting $\gamma_0$ and $\gamma_r$ passes within a bounded
marking distance of any $K$-intermediate train track.  As the marking
complex is locally finite, infinitely many train track sequences must
share a common train track, and this implies that infinitely many of
the handlebodies $V_i$ must share a common disc $D_r$. Furthermore,
the boundary of $D_r$ lies close to the geodesic from $\gamma_0$ to
$\gamma_r$.

We iterate this argument to obtain the following.  Choose an
increasing sequence of numbers $\{ r_n \}$, with the difference
between consecutive numbers fixed and sufficiently large.  At each
stage, we may assume we have passed to an infinite subset of the
handlebodies so that each contains the discs
$D_{r_1}, \ldots D_{r_n}$.  These discs lie in a common compression
body $W_n$, and the geodesic from $\gamma_0$ to $\gamma_{r_n}$ is
contained in a bounded neighbourhood of $\mathcal{D}(W_n)$.  The
compression bodies $W_n$ form an increasing sequence $W_n < W_{n+1}$,
which eventually stabilizes to a constant sequence $W$.  The entire
geodesic ray $\gamma$ is then contained in a bounded neighbourhood of
$\mathcal{D}(W)$, by the stability result.

In the next section, Section \ref{section:prelim}, we review the
results we will use, and set up some notation.  In Section
\ref{section:stability}, we provide the details of the proof of the
stability result.  In Section \ref{section:diameter}, we use the
stability result to prove Theorem \ref{Thm:InfDiam}.  In Section
\ref{section:loxodromics}, we show that there are many loxodromic
isometries of the compression body graph.  In the final section,
Section \ref{section:applications}, we give several applications.

\section{Preliminaries} \label{section:prelim}

In this section we review some background material, and set up some
notation.

\subsection{The mapping class group}

Let $S$ be a compact connected oriented surface, possibly with
boundary.  The \emph{mapping class group} $\text{Mod}(S)$ is the group
of homeomorphisms of $S$, up to isotopy.  We shall fix a finite
generating set for the mapping class group, and we shall write
$\dg{g,h}$ for the corresponding word metric on $\text{Mod}(S)$.

We say a simple closed curve in $S$ is \emph{peripheral} if it
cobounds an annulus together with one of the boundary components of
$S$.  We say a simple closed curve in $S$ is \emph{essential} if it
does not bound a disc and is not peripheral.

Given two finite collections of essential curves $\mu$ and $\mu'$, we
may extend the geometric intersection number from single curves to
finite collections by
\[ i(\mu, \mu') = \sum_{a \in \mu, \, a' \in \mu'} i(a, a'). \]
We say a set of curves $\mu$ is a \emph{marking} of $S$ if the set of
curves \emph{fills} the surface: that is, for all curves $a \in \C(S)$
we have $i(a, \mu) > 0$.  We say a marking $\mu$ is an
\emph{$L$-marking} if $i(\mu, \mu) \le L$.  We say two markings $\mu$
and $\mu'$ are $L'$-adjacent if $i(\mu, \mu') \le L'$.  Define
$\calM_{L, L'}(S)$ to be the graph whose vertices are (isotopy classes
of) $L$-markings and whose edges connect pairs of $L$-markings which
are $L'$-adjacent.  Masur and Minsky \cite{mm2}*{Section 7.1} show
that, for sufficiently large constants $L$ and $L'$, the \emph{marking
  graph} $\calM(S) = \calM_{L, L'}(S)$ is locally finite, connected
and quasi-isometric to the Cayley graph of the mapping class group
$(\text{Mod}(S), d_\text{Mod})$. We shall fix a pair of constants $L$
and $L'$ with this property for the remainder of this paper, and will
suppress these constants from our notation.

The \emph{complex of curves} $\C(S)$ is a simplicial complex, whose
vertices consist of isotopy classes of essential curves, and whose
simplices are spanned by disjoint (perhaps after isotopy) curves.  We
need to modify this definition for certain low-complexity surfaces, as
we now describe.  In the case of a once-holed torus we connect two
curves by an edge if they have geometric intersection number one.  In
the case of a four-holed sphere we connect two curves by an edge if
they have geometric intersection number two.  Finally, if the surface
$S$ is an annulus $A$, then we define the curve complex $\C(A)$ as
follows: the vertices consist of properly embedded essential arcs up
to homotopies fixing the endpoints, and two arcs are connected by an
edge if they may be realized disjointly in the interior of the
annulus.  We shall consider the complex of curves as a metric space in
which each edge has length one.  We will write $d_S$ for distance in
the complex of curves.  We extend the definition of the metric from
points to finite sets by setting
\[ d_S(A, B) = \min_{a \in A, \, b \in B} d_S(a, b). \]

Suppose that $S$ is not an annulus.  Then a subsurface $Y \subset S$
is \emph{peripheral} if it is an annulus with peripheral boundary.  A
subsurface $Y \subset S$ is \emph{essential} if it is not peripheral
and if all boundary components are essential or peripheral.  Let
$Y \subset S$ be an essential subsurface. We will write $d_Y$ for the
metric in $\C(Y)$. Let $Y_\varnothing$ be the set of vertices of
$\C(S)$ corresponding to essential curves which may be isotoped to be
disjoint from $Y$. There is a coarsely well defined \emph{subsurface
  projection} $\pi_Y \colon \C(S) \setminus Y_\varnothing \to \C(Y)$,
which we now define.  We say a properly embedded arc in a surface $S$
is \emph{essential} if it does not bound a properly embedded bigon
together with a subarc of the boundary of $S$.  Let $\mathcal{AC}(S)$
be the \emph{arc and curve complex} of $S$: the simplicial complex
whose vertices are isotopy classes of essential curves and properly
embedded essential simple arcs.  Thus $\mathcal{AC}(S)$ contains
$\C(S)$ as a subcomplex, and this inclusion is a quasi-isometry.  Let
$S^Y$ be the cover of $S$ corresponding to $Y$.  The surface $Y$ is
homeomorphic to the Gromov closure of $S^Y$ so we may identify
$\mathcal{AC}(Y)$ with $\mathcal{AC}(S^Y)$.  For any essential curve
or arc $a$, let $a^Y$ be the full preimage of $a$ in $S^Y$.  Define
$\pi_Y(a)$ to be the set of essential components of $a^Y$ in $S^Y$.
So $\pi_Y(a)$ is either empty, or is a simplex in $\mathcal{AC}(Y)$.
As $\mathcal{AC}(Y)$ is quasi-isometric to $\C(Y)$, this gives a
coarsely well-defined map to $\C(Y)$ in the latter case.

We define the \emph{cut-off function} $\lfloor \cdot \rfloor_c$ on
$\R$ by $\lfloor x \rfloor_c = x$ if $x \ge c$ and zero otherwise.
Given a set $X$, two functions $f$ and $g$ on $X \times X$, and two
constants $K \ge 1$ and $c > 0$, we say that $f$ and $g$ are
$(K, c)$-coarsely equivalent, denoted $f \approx_{(K, c)} g$, if for
all $x$ and $y$ in $X$ we have
\[ \tfrac{1}{K} f(x, y) - \tfrac{c}{K} \le g(x, y) \le K f(x, y) + c. \]

We now state the \emph{distance estimate} due to Masur and Minsky.

\begin{theorem} \cite{mm2}*{Theorem 6.12} \label{theorem:mm dist} %
For any surface $S$ there is a constant $M_0$, such that for any $M
\ge M_0$, there are constants $K$ and $c$, such that for
any markings $\mu$ and $\nu$,
\[
d_{\M(S)}( \mu, \nu) \approx_{(K, c)} \sum_X \lfloor d_X( \mu, \nu )
\rfloor_M.
\]
\end{theorem}

That is, the distance in the marking complex is coarsely equivalent to
the (cut-off) sum of subsurface projections.  Note that there are only
finitely many non-zero terms on the right-hand side.

\subsection{Compression bodies and discs sets}

Let $S$ be a closed connected oriented surface of genus $g$.  A
\emph{compression body} $V$ is a compact orientable three-manifold
obtained from $S \cross I$ by attaching two-handles to
$S \cross \{ 0 \}$, and capping off any newly created two-sphere
components with three-balls.  In particular, if $S$ is a two-sphere,
we do not cap off $S \cross \{ 1 \}$.
The genus $g$ surface $S \cross \{ 1 \}$ is the \emph{upper boundary}
$\partial_+ V$, while the other boundary components make up the
\emph{lower boundary} $\partial_- V$.  The lower boundary need not be
connected. If $\partial_- V$ is empty, then $V$ is a
\emph{handlebody}.  The \emph{trivial compression body} has no
two-handles, and is homeomorphic to $S \cross I$.

A \emph{marked compression body} is a pair $(V, f )$ where $V$ is a
compression body and $f \colon \partial_+ V \to S$ is a homeomorphism.
Two marked compression bodies $(V, f)$ and $(W, g)$ are
\emph{isomorphic} if there is a homeomorphism $h \colon V \to W$ such
that $g \circ (h |_{\partial V}) = f $.  In what follows, we will
suppress the marking $f$ and assume that $\bdy V$ and $S$ are actually
equal.

Biringer and Vlamis define the \emph{compression body graph}
$\mathcal{H}(S)$ to be the graph whose vertices are isomorphism
classes of non-trivial marked compression bodies.  Here $V$ and $W$
are adjacent if either $V < W$ or $W < V$ \cite{bv}*{page 94}.  They
show that the simplicial automorphism group of $\mathcal{H}(S)$ is
equal to the mapping class group $\text{Mod}(S)$ \cite{bv}*{Theorem
  1.1} for genus at least three.  In genus two, the mapping class
group surjects onto the automorphism group, with kernel generated by
the hyperelliptic involution.

We say a disc $D$ in a marked compression body $(V, f)$ is
\emph{essential} if it is properly embedded in $V$.  In this case, its
boundary is an essential curve in the upper boundary $\partial_+ V$.
We say an essential curve in $S$ bounds a disc in $V$ if it is the
image under $f$ of the boundary of an essential disc in $V$.

\begin{definition}
The \emph{disc set} $\mathcal{D}(V)$ of a marked compression body
$(V, f)$ is the collection of all essential curves in $S$ which bound
discs in $V$.
\end{definition}

A subset $Y$ of a geodesic metric space is \emph{$Q$-quasi-convex} if
for any pair of points $x$ and $y$ in $Y$, any geodesic connecting $x$
and $y$ is contained in a $Q$-neighbourhood of $Y$.  Masur and Minsky
\cite{mm3}*{Theorem 1.1} showed that there is a $Q$ such that the disc
set $\mathcal{D}(V)$ is a $Q$-quasi-convex subset of $\C(S)$.  Given a
metric space $(X, d)$, and a collection of subsets
$Y = \{ Y_i \}_{i \in I} $, the \emph{electrification} of $X$ with
respect to $Y$ is the metric space $X_\calY$ obtained by adding a new
vertex $y_i$ for each set $Y_i$, and coning off $Y_i$ by attaching
edges of length $\tfrac{1}{2}$ from each $y \in Y_i$ to $y_i$; the
image of each set $Y_i$ in $X_\calY$ has diameter one.  We shall write
$\mathcal{D}(S)$ to denote the collection of all disc sets of all
non-trivial compression bodies $V$ with boundary $S$.  The compression
body graph $\CB(S)$ is quasi-isometric to the curve complex
electrified along all the disc sets, namely $\C(S)_{\mathcal{D}(S)}$.

We remark that there is another natural space quasi-isometric to
$\C(S)_{\calD(S)}$, known as the \emph{graph of handlebodies}, which
has vertices being classes of marked handlebodies and edges being
distinct pairs $\{ V , W\}$ where $\mathcal{D}(V)$ has non-empty
intersection with $\calD(W)$.

\subsection{Surgery sequences}

In this section we recall the definition of surgery sequences for
simple closed curves and for discs. Since we restrict our attention to
\emph{closed} surfaces our discussion is simpler than the more general
case of compact surfaces with boundary.  In subsequent sections we may
abuse notation by referring to these as just surgery sequences, if it
is clear from context whether we mean curves or discs.

\begin{definition}
Let $a$ be an essential curve in $S$, and let $b'$ be a simple arc
whose endpoints lie on $a$, and whose interior is disjoint from $a$.
Furthermore, suppose that $b'$ is essential in $S \setminus a$.  The
endpoints of $b'$ divide $a$ into two arcs with common endpoints, $a'$
and $a''$, say.  A simple closed curve $c$ is said to be produced by
\emph{(arc) surgery of $a$ along $b'$} if $c$ is homotopic to either
of the simple closed curves $a' \cup b'$ or $a'' \cup b'$.
\end{definition}

\begin{definition}
Let $a$ and $b$ be essential simple closed curves in minimal position
with $i(a, b) \ge 2$.  An \emph{innermost arc} of $b$ with respect to
$a$ is a subarc $b' \subset b$ whose endpoints lie on $a$, and whose
interior is disjoint from $a$.  A simple closed curve $c$ is said to
be produced by \emph{(curve) surgery of $a$ along $b$} if $c$ is
produced by surgery of $a$ along $b'$, for some choice of innermost
arc $b'$ in $b$.
\end{definition}

If $c$ is produced by surgery of $a$ along $b$, then the number of
intersections of $c$ with $b$ is strictly less than the number of
intersections of $a$ with $b$.

\begin{definition}
Given a pair of essential simple closed curves $a$ and $b$, a
\emph{(curve) surgery sequence} connecting $a$ and $b$ is a sequence
of simple closed curves $\{ a_i \}_{i = 0}^{n-1}$, such that
$a_0 = a$, the final curve $a_{n-1}$ is disjoint from $b$, and each
$a_{i+1}$ is produced from $a_{i}$ by a surgery of $a_i$ along $b$.
\end{definition}

\begin{definition}
Let $D$ and $E$ be essential discs in a compression body $V$ in
minimal position, and let $E'$ be an outermost bigon of $E$ with
respect to the arcs of $D \cap E$.  The arc of intersection between
$E'$ and $D$ divides $D$ into two discs, $D'$ and $D''$, say.  We say
that a disc $F$ is produced by \emph{(disc) surgery of $D$ along $E'$}
if $F$ is homotopic to either of the discs $D' \cup E'$ or
$D'' \cup E'$, for some choice of outermost bigon $E'$ contained in
$E$.
\end{definition}

The disc $F$ produced by surgery of $D$ along $E'$ is disjoint from
$D$, and is essential.  To see this, suppose that one of the discs $F$
is inessential. That is, $\partial F$ bounds a disc $C$ in $S$.  We
can then ambiently isotope $\partial E' \cap S$ across $C$, reducing
the number of intersections between $\partial D$ and $\partial E$, a
contradiction.

\begin{definition}
Given a pair of essential discs $D$ and $E$ in minimal position, a
\emph{(disc) surgery sequence} connecting $D$ and $E$ is a sequence of
essential properly embedded discs $\{ D_i \}_{i = 0}^{n-1}$, such that
$D_1 = D$, the final disc $D_{n-1}$ is disjoint from $E$, and each
$D_{i+1}$ is produced from $D_{i}$ by a (disc) surgery of $D_i$ along
$E$.
\end{definition}

We remark that if $\{ D_i \}_{i=1}^n$ is a (disc) surgery sequence,
then $\{ \partial D_i \}_{i=1}^n$ is a (curve) surgery sequence.

\begin{proposition}
Let $V$ be a compression body.  Any two curves in its disc set
$\mathcal{D}(V)$ are connected by a disc surgery sequence.
\end{proposition}

\begin{proof}
Let $D$ and $E$ be two essential discs in a compression body $V$, and
assume they intersect in minimal position.  Choose an outermost bigon
$E'$ of $E$ with respect to the arcs $D \cap E$. We may surger the
disc $D$ along $E'$, which produces two new discs in $V$ disjoint from
$D$.  Call one of these $D_1$.  The disc $D_1$ is disjoint from $D$
and has fewer intersections with $E$.  This process terminates after
finitely many disc surgeries.  This gives a disc surgery sequence
$\{ D_i \}_{i=0}^{n-1}$ connecting $D$ and $E$.
\end{proof}

\subsection{Train tracks}\label{section:tt}

We briefly recall some of the results we use about train tracks on
surfaces.  For more details see for example Penner and Harer
\cite{ph}.

A \emph{pre-train track} $\tau$ is a smoothly embedded finite graph in
$S$ such that the edges at each vertex are all mutually tangent, and
there is at least one edge in each of the two possible directed
tangent directions. The vertices are commonly referred to as
\emph{switches} and the edges as \emph{branches}.  We will always
assume that all switches have valence at least three.  A trivalent
switch is illustrated below in Figure \ref{fig:switch}.  If none of
the complementary regions of $\tau$ in $S$ are nullgons, monogons,
bigons or annuli, then we say that $\tau$ is a \emph{train track}.  Up
to the action of the mapping class group, there are only finitely many
train tracks in $S$.

\begin{figure}[H]
\begin{center}
\begin{tikzpicture}[scale=0.3]

\draw [thick] (20,13) node [left] {$a$} -- 
              (22.5, 13) .. controls (23.5,13) and (25,13) .. 
              (28,10) node [right] {$b$};
\draw [thick] (20,13) -- 
              (22.5, 13) .. controls (23.5,13) and (25,13) .. 
              (28,16) node [right] {$c$};

\end{tikzpicture}
\end{center}
\caption{A trivalent switch for a train track.} \label{fig:switch}
\end{figure}
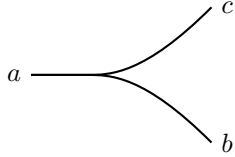

An assignment of non-negative numbers to the branches of $\tau$, known
as \emph{weights}, satisfies the \emph{switch equality} if the sum of
weights in each of the two possible directed tangent directions is
equal: that is $a = b + c$ in Figure \ref{fig:switch} above. A
weighted train track defines a measured lamination on the surface.  We
say that the corresponding lamination is \emph{carried} by the train
track.  A train track $\tau$ determines a polytope of projectively
measured laminations $P(\tau) \subset \PML(S)$ carried by $\tau$. Let
$\Lambda(\tau)$ be the set of vertices of $P(\tau)$. Every
$v \in \Lambda(\tau)$ gives a \emph{vertex cycle}: a simple closed
curve, carried by $\tau$, that puts weight at most two on each branch
of $\tau$. It follows that the set $\Lambda(\tau)$ gives a finite set
of curves in $\C(S)$ and these curves have bounded pairwise geometric
intersection numbers.  A train track $\tau$ is \emph{filling} if
$\Lambda(\tau)$ is a marking.  As there are only finitely many train
tracks up to the action of the mapping class group, these bounds
depend only on the surface $S$.

A simple closed curve $a$ is \emph{dual} to a track $\tau$ if $a$
misses the switches of $\tau$, crosses the branches of $\tau$
transversely, and forms no bigons with $\tau$.  We say that a simple
closed curve $a$ is \emph{switch-dual} to $\tau$ if $a$ meets the
train track transversely at exactly one point of $\tau$, which is a
switch, and forms no bigons with $\tau$.

A track $\tau$ is \emph{large} if all components of $S \setminus \tau$
are discs or peripheral annuli. A track $\tau$ is \emph{recurrent} if
for every branch $b \subset \tau$ there is a curve
$\alpha \in P(\tau)$ putting positive weight on $b$. A track $\tau$ is
\emph{transversely recurrent} if for every branch $b \subset \tau$
there is a curve $\beta$ dual to $\tau$, such that $\beta$ crosses
$b$.

\begin{lemma}
\label{Lem:TracksFinite}
Given a surface $S$, there is a constant $N$, with the following
property.  For any marking $\mu$ of $S$, there are at most $N$
non-isotopic filling train tracks $\tau$ with $\Lambda(\tau) = \mu$.
\end{lemma}

\begin{proof}
There are only finitely many isotopy classes of filling train tracks
up to the action of the mapping class group.  If $\tau = g \sigma$,
for some element of the mapping class group $g \in \text{Mod}(S)$, and
$\Lambda(\tau) = \Lambda(\sigma) = \mu$, then $g$ preserves the
collection of curves $\mu$.  Note that the collection of curves $\mu$
fills $S$ and has finite total self-intersection number depending only
on $S$.  We deduce that there are only finitely many such mapping
classes $g$.
\end{proof}

A \emph{split} of a train track $\tau$ produces a new train track
$\sigma$ by one of the local modifications illustrated in Figure
\ref{pic:split} below.  Here a subset of $\tau$ diffeomorphic to the
top configuration, is replaced with one of the lower three
configurations.

\begin{figure}[H]
\begin{center}
\begin{tikzpicture}[scale=0.3]

\def\tangent{
\draw[thick] (-2, 1) .. controls (-1.5, 0.25) and (-1, 0) .. (0, 0);
}

\def\segment{
\tangent
\draw[thick] (0, 0) -- (4, 0);
\begin{scope}[xshift=4cm,x=-1cm]
        \tangent
\end{scope}
}

\segment
\begin{scope}[y=-1cm]
        \segment
\end{scope}

\def\split{
\begin{scope}[yshift=-5cm]
        \segment
\end{scope}

\begin{scope}[yshift=-6cm,y=-1cm]
        \segment
\end{scope}
}

\split

\begin{scope}[xshift=-10cm]
        \split
        \draw[thick] (0, -6) .. controls (2, -6) and (2, -5) .. (4, -5);
\end{scope}

\begin{scope}[xshift=10cm]
        \split
        \draw[thick] (0, -5) .. controls (2, -5) and (2, -6) .. (4, -6);
\end{scope}

\end{tikzpicture}%
\end{center} 
\caption{Splitting a train track.} \label{pic:split}
\end{figure}
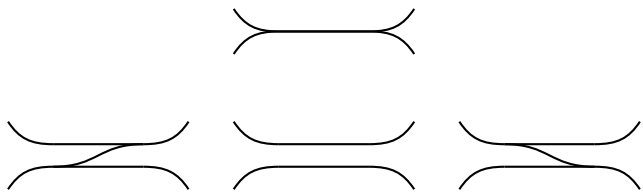

A \emph{shift} for a train track is the local modification given in
Figure \ref{pic:shift}.

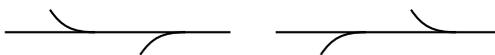
\begin{figure}[H]
\begin{center}
\begin{tikzpicture}[scale=0.3]

\def\tangent{
\draw[thick] (-2, 1) .. controls (-1.5, 0.25) and (-1, 0) .. (0, 0);
}

\def\segment{
\tangent
\draw[thick] (0, 0) -- (4, 0);
\begin{scope}[xshift=4cm,x=-1cm]
        \tangent
\end{scope}
}

\tangent

\draw [thick] (-4, 0) --  (6, 0);

\begin{scope}[y=-1cm, xshift=4cm]
        \tangent
\end{scope}

\begin{scope}[xshift=12cm]

    \begin{scope}[y=-1cm]
        \tangent
    \end{scope}

    \draw [thick] (-4, 0) --  (6, 0);

    \begin{scope}[xshift=4cm]
        \tangent
    \end{scope}

\end{scope}

\end{tikzpicture}%
\end{center} 
\caption{Shifting for a train track.} \label{pic:shift}
\end{figure}

A \emph{tie neighbourhood} $N(\tau)$ for a train track $\tau$ is a
union of rectangles, as follows.  For each switch there is a rectangle
$R(s)$, and for each branch $b$ there is a rectangle $R(b)$. All
rectangles are foliated by vertical arcs called \emph{ties}.  We glue
a vertical side of a branch rectangle to a subset of a vertical side
of a switch rectangle as determined by the combinatorics of $\tau$.
See Figure \ref{fig:tie} for a local picture near a trivalent switch.

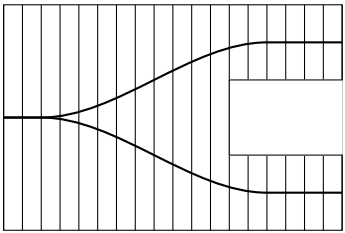
\begin{figure}[H]
\begin{center}
\begin{tikzpicture}[scale=0.5]

\foreach \x in {0,0.5,...,9} { \draw (\x, 0) -- (\x, 6); }; 

\filldraw [white] (6, 2) rectangle (9,4);

\draw (0, 0) -- (9, 0) -- (9, 2) -- (6, 2) -- (6, 4) -- (9, 4) -- (9,
6) -- (0, 6) -- cycle;

\draw [thick] (0, 3) -- (1, 3) .. controls  (3, 3) and (5, 1) .. (7,
1) -- (9, 1);
\draw [thick] (0, 3) -- (1, 3) .. controls  (3, 3) and (5, 5) .. (7,
5) -- (9, 5);
    
\end{tikzpicture}
\end{center}
\caption{A tie neighbourhood for a train track.} \label{fig:tie}
\end{figure}

We say a train track or simple closed curve $\sigma$ is \emph{carried}
by $\tau$ if $\sigma$ may be isotoped to lie in $N(\tau)$, such that
$\sigma$ is transverse to the ties of $N(\tau)$.  We denote this by
either $\sigma \carried \tau$ or $\tau \carries \sigma$.  If $\sigma$
is a train track obtained by splitting and shifting $\tau$ then
$\tau \carries \sigma$. A \emph{train track carrying sequence} is a
sequence of train tracks $\tau_0 \carries \tau_1 \carries \cdots$,
such that each $\tau_{i}$ carries $\tau_{i+1}$.  We may also denote a
train track splitting sequence by $\{ \tau_i \}_{i = 0}^n$, where
$n \in \N_0 \cup \{ \infty \}$.  We say that $\{ \tau_i \}_{i = 0}^n$
is \emph{$K$-connected} if
\[ d_S( \Lambda(\tau_i), \Lambda(\tau_{i+1})) \le K\]
for all $i$.

Masur and Minsky \cite{mm3}*{Theorem 1.3} showed that train track
splitting sequences give rise to reparameterized quasi-geodesics in
the curve complex $\C(S)$.  Using this, Masur, Mosher and Schleimer
show that train track splitting sequences give reparameterized
quasi-geodesics under subsurface projection.

\begin{theorem} \cite{mms}*{Theorem 5.5} \label{theorem:qg} %
For any surface $S$ with $\xi(S) \ge 1$ there is a constant $Q = Q(S)$
with the following property: For any sliding and splitting sequence
$\{ \tau_i \}^n_{i=0}$ of birecurrent train tracks in $S$ and for any
essential subsurface $X \subset S$, if $\pi_X( \tau_i ) \not =
\varnothing$ then the sequence $\{ \pi_X ( \Lambda ( \tau_i ) ) \}^n_{i=0}$
is a $Q$-reparameterized quasi-geodesic in the curve complex $\C(X)$.
\end{theorem}

We abuse notation by writing $d_X(\tau , \cdot)$ for
$d_X(\pi_X(\Lambda(\tau)), \cdot)$, where $X$ is an essential
subsurface of $S$.  Given a train track carrying sequence
$\{ \tau_i \}_{i = 0}^n$, we say that a train track $\tau_i$ in the
sequence is \emph{$K$-intermediate} if $d_S(\tau_0, \tau_i) \ge K$ and
$d_S(\tau_i, \tau_n) \ge K$.  We will use the following consequence of
Theorem \ref{theorem:qg}:

\begin{theorem} \label{Thm:MMS} %
There is a constant $B$, depending only on $S$, such that for any
constant $A$, and any two carrying sequences $\{ \tau_i \}_{i = 0}^n$
and $\{ \sigma_i \}_{i = 0}^{m}$, with $d_S(\tau_0, \sigma_0) \le A$
and $d_S(\tau_n, \sigma_{m}) \le A$, then for any pair of
$(A + B)$-intermediate train tracks $\tau_i$ and $\sigma_j$, and any
subsurface $X \subset S$,
\[ d_X(\tau_i, \sigma_j) \le d_X(\tau_0, \tau_n) + B. \]
\end{theorem}

To prove Theorem \ref{Thm:MMS} we will also need the following
\emph{bounded geodesic image theorem} of Masur and Minsky.

\begin{theorem}\cite[Theorem 3.1]{mm2}\label{theorem:bounded image}
For any surface $S$, and for any essential subsurface $X$ of $S$,
there is a constant $M_X$, such that for any geodesic $\gamma$ in
$\mathcal{C}(S)$, all of whose vertices intersect $X$ non-trivially,
the projected image of $\gamma$ in $\mathcal{C}(X)$ has diameter at
most $M_X$.
\end{theorem}

\begin{proof}[Proof of Theorem \ref{Thm:MMS}]
If there is a constant $B$, depending on $S$, such that
$d_X(\tau_i, \sigma_j) \le B$, then we are done.  Otherwise, by the
bounded geodesic image theorem, $\partial X$ is also
$(A+B)$-intermediate.  Then $d_X(\tau_0, \sigma_0) \le M_X$ and
$d_X(\tau_n, \sigma_m) \le M_X$.

By Theorem \ref{theorem:qg}, there is a constant $Q$, such that the
images of both $\{ \Lambda(\tau_i) \}_{i=1}^n$ and
$\{ \Lambda(\sigma_i) \}_{i=1}^{m}$ under $\pi_X$ are reparameterized
$Q$-quasi-geodesics, and furthermore, their endpoints are distance at
most $A_2$ apart in $\mathcal{C}(X)$.  By the Morse property for
quasi-geodesics, there is a constant $A_3$, depending only on $M_X, Q$
and $\delta$, and hence only on the surface $S$, such that
$\{ \Lambda(\tau_i) \}_{i=1}^n$ and
$\{ \Lambda(\sigma_i) \})_{i=1}^{m}$ are Hausdorff distance at most
$A_3$ apart, and so the distance between $\pi_X(\tau_i)$ and
$\pi_X(\sigma_j)$ is at most $d_X(\tau_0, \tau_n) + A_3$.  Therefore
the result follows, choosing $B = \max \{ A_1 + O(\delta), A_3\}$,
which only depends on the topology of the surface $S$.
\end{proof}

We will abuse notation and say that an essential disc $E$ is
\emph{carried} by a train track $\tau$ if $\partial E$ is carried by
$\tau$.

We will use the following result of Masur and Minsky \cite{mm3}*{page 309}.

\begin{proposition} \label{prop:splitting seq} %
Let $D$ and $E$ be essential discs contained in a compression body
$V$. Then there is a surgery sequence $\{D_i\}_{i = 0}^{n-1}$ with
$D = D_0$ and $D_{n-1}$ disjoint from $E$, and a carrying sequence
$\{\tau_i\}_{i = 0}^n$, with the following properties.

\begin{enumerate}

\item The train track $\tau_i$ has only one switch for all
$0 \le i \le n$.

\item The boundary of $D_i$ is switch-dual to the track $\tau_i$, for
all $0 \le i < n$.

\item The disc $E$ is carried by the train track $\tau_i$, for all
$0 \le i \le n$.

\end{enumerate}

\end{proposition}

Masur and Minsky \cite{mm3} work in a more general setting allowing
for surfaces with boundary components, and state a weaker version of
property 2, that $d_S(D_i, \tau_i) \le 5$. However, in the case of
closed surfaces, the statement we need follows from the proof of
\cite{mm3}*{Lemma 4.1}.  We provide a review of their work in the
appendix for the convenience of the reader.

Finally, we observe:

\begin{proposition} \label{prop:large} %
For all surfaces $S$, there is a constant $K$ such that for any curves
$\partial D$ and $\partial E$ in $\mathcal{D}$, and any carrying
sequence $\{ \tau_i \}_{i=0}^n$ connecting them, any $K$-intermediate
train track in $\{ \tau_i \}_{i=0}^n$ is birecurrent and filling.
\end{proposition}

\begin{proof}
For all $i$, the curve $\partial E$ runs over every edge of $\tau_i$,
so $\tau_i$ is recurrent. The train track $\tau_i$ has only one
switch, and $\partial D_i$ crosses $\tau_i$ exactly once at that
switch, so $\tau_i$ is transversely
recurrent. 

If $\tau_i$ is not filling, then there is a curve $a$ disjoint from
the set of vertex cycles $\Lambda(\tau_i)$.  Then $a$ is disjoint from
every collection of simple closed curves produced from positive
integer valued sums of the $\Lambda(\tau_i)$, and in particular
$\tau_i$ cannot carry a pair of curves which fill $S$.  However, the
train track $\tau_i$ carries $\partial E$. As long as
$d_S(\Lambda(\tau_i), \partial E) \ge 3$, there is a vertex cycle
distance at least $3$ from $\partial E$, and so $\tau_i$ carries a
pair of curves which fill $S$, and so $\Lambda(\tau_i)$ fills.
Therefore a $K$-intermediate train track $\tau_i$ is filling, for any
$K \ge 3$.
\end{proof}

\subsection{Laminations}

If $f$ is a pseudo-Anosov homeomorphism of $S$ which extends over a
compression body $V$, then the stable lamination of $f$ is a limit of
discs of $V$. On the other hand we have the following:

\begin{proposition} \label{prop:lamination}
There is a minimal lamination in $\PML(S)$ which does not lie in the
limit set of a disc set $\mathcal{D}(V)$, for any compression body
$V$.
\end{proposition}

\begin{proof}
Recall that $\PML(S)$ is the space of projectively measured
laminations in $S$. Fix a handlebody $V$. Define the \emph{limit set}
$\overline{\calD}(V)$ of $V$ to be the closure of $\calD(V)$,
considered as a subset of $\PML(S)$.  Following Masur
\cite{masur}*{Theorem 1.2}, we define $Z(V)$, the \emph{zero set} of
$\overline{\calD}(V)$, to be the set of laminations $\mu \in \PML(S)$
where there is some $\nu \in \overline{\calD}(V)$ having geometric
intersection $i(\mu, \nu)$ equal to zero. Note that $Z(V)$ is closed
in $\PML(S)$. Masur \cite{masur}*{Theorem 1.2} shows that $Z(V)$ has
empty interior in $\PML(S)$, and Gadre \cite{gadre}*{Theorem A.1}
shows that it has measure zero.  Therefore, the complement of the
countable union $\cup_V Z(V)$ is full measure in $\PML(S)$. The set of
minimal laminations also has full measure in $\PML(S)$, so there is at
least one minimal lamination disjoint from the union of the limit sets
of the disc sets.
\end{proof}

We remark that this argument also works using harmonic measure instead
of Lebesgue measure by work of Kaimanovich and Masur
\cite{km}*{Theorem 2.2.4} and Maher \cite{maher_heegaard}*{Theorem
  3.1}.

\section{Stability} \label{section:stability}

\begin{theorem}
\label{Thm:Ray}
Suppose $\gamma \from \N \to \C(S)$ is a geodesic ray. The diameter of
the image of $\pi_{\mathcal{H}(S)} \circ \gamma$ is finite if and
only if there is a compression body $V$ and a constant $k$ so that the
image of $\gamma$ lies in a $k$-neighborhood of $\calD(V)$.
\end{theorem}

The backward direction is immediate. In this section we show the
forward direction.  A standard coarse geometry argument using the
hyperbolicity of the curve complex, plus the quasi-convexity of
$\calD(V)$ inside of $\C(S)$, reduces the forward direction of Theorem
\ref{Thm:Ray} to the following statement, which we call the stability
hypothesis.

\begin{theorem}[Stability hypothesis]
\label{Thm:Stability}
Given a surface $S$ and a constant $k$, there is a constant $k' \ge k$
with the following property.  Suppose that $\gamma$ is a geodesic ray
in $\C(S)$, and $V_i$ is a sequence of compression bodies such that,
for all $i$, the segment $\gamma|[0,i]$ lies in a $k$--neighborhood of
$\calD(V_i)$.  Then there is a constant $k'$ and a non-trivial
compression body $W$, contained in infinitely many of the $V_i$, such
that that $\gamma$ is contained in a $k'$-neighborhood of $\calD(W)$.
\end{theorem}

We first show the following.

\begin{lemma} \label{lemma:close} %
There is a constant $K$, which depends on $S$, such that for any two
essential simple closed curves $a$ and $b$ in $S$, with
$d_S(a, b) \ge 3K$, there is a constant $N(a, b, K)$, such that for
any collection $\{ V_i \}_{i=1}^N $ of compression bodies with
$d_S(a, \mathcal{D}(V_i)) \le K$ and $d_S(b, \mathcal{D}(V_i)) \le K$
there are at least two compression bodies $V_i$ and $V_j$ which share
a common simple closed curve $c$.  Furthermore, the curve $c$ is
$K$-intermediate for $a$ and $b$.
\end{lemma}

For fixed $K$, the constant $N$ is coarsely equivalent to the smallest
marking distance between any markings containing $a$ and $b$.

\begin{proof}[Proof of Lemmma \ref{lemma:close}]
For each compression body $V_i$, choose curves $a_i$ and $b_i$ in
$\mathcal{D}(V)$ such that $d_S(a, a_i) \le K$ and
$d_S(b, b_i) \le K$. The discs bounded by $a_i$ and $b_i$ in $V_i$
determine a disc surgery sequence $\{ D^i_j \}_{j=0}^{n_i - 1}$, with
$D^i_0 = a_i$ and $D^i_{n_i - 1}$ disjoint from $b_i$, and a train
track carrying sequence $\{ \tau^i_j \}_{j=0}^{n_i - 1}$.

By Proposition \ref{prop:large} there is a $K$ such that for all
$i, j$, every $K$-intermediate train track $\tau^i_j$ is filling and
birecurrent.  By Theorem \ref{Thm:MMS}, and bounded geodesic
projections, there is a constant $K_1$ such that for all
$K$-intermediate train tracks $\tau^{i_1}_{j_1}$ and
$\tau^{i_2}_{j_2}$, and any subsurface $Y \subset S$,
\[ d_Y( \tau^{i_1}_{j_1}, \tau^{i_2}_{j_2}) \le d_Y(a, b ) + K_1.  \]
Therefore, using the Masur-Minsky distance formula, Theorem
\ref{theorem:mm dist}, with cutoff $M$ larger than $K_1 + M_0$, there
is a constant $N_1$ such that
\[ d_{\mathcal{M}}( \tau^{i_1}_{j_1} , \tau^{i_2}_{j_2} ) \le
N_1.  \]
By Lemma \ref{Lem:TracksFinite} there are most $N_2$ train tracks
$\tau$ with $\Lambda(\tau) = \mu$, for any marking $\mu$.  As
$d_{\mathcal{M}}$ is a proper metric, for any $K$-intermediate track
$\tau^i_j$, there are at most
$N_3 = N_2 \norm{ B_{\mathcal{M}}(\Lambda(\tau^i_j), N_1) }$ train
tracks with markings within distance $N_1$ of $\Lambda(\tau^i_j)$.
Thus, if there are at least $N_3 + 1$ compression bodies, at least two
of them must share a common train track.

For each train track $\tau^i_j$, there is a unique simple closed curve
that is switch-dual to $\tau^i_j$, and bounds a disc in the
compression body $V_i$. Therefore, if there are at least $N = N_3 + 1$
compression bodies, at least two of them must have a disc in common.
\end{proof}

We now complete the proof of Theorem \ref{Thm:Stability}.

\begin{proof}[Proof of Theorem \ref{Thm:Stability}]
Choose a subsequence $(n_k)_{k \in \N_0}$ with $n_0 = 0$ such that
$d_S(\gamma(n_{k}), \gamma(n_{k+1}) ) \ge 3 K$ for all $k$.

By Lemma \ref{lemma:close}, there is an infinite subset of the
compression bodies $V_i$ such that the $V_i$ contain a simple closed
curve $\partial D_0$ which is $K$-intermediate for $\gamma(n_0)$ and
$\gamma(n_1)$. We may pass to this infinite subset, and then apply
Lemma \ref{lemma:close} to $\gamma(n_1)$ and $\gamma(n_2)$, producing
a simple closed curve $\partial D_1$, which is $K$-intermediate for
$\gamma(n_1)$ and $\gamma(n_2)$, and such that $\partial D_0$ and
$\partial D_1$ simultaneously compress in infinitely many compression
bodies.

Isotope $\partial D_0$ and $\partial D_1$ in $S$ so they realize their
geometric intersection number.  By work of Casson and Long \cite[Proof
of Lemma 2.2]{casson-long}, the curves $\partial D_0$ and
$\partial D_1$ simultaneously compress in a handlebody $V$ if and only
if there is a pairing on the points of
$\partial D_0 \cap \partial D_1$ that is simultaneously unlinked on
$\partial D_0$ and on $\partial D_1$.  By the previous paragraph, we
know that there is at least one such pairing. There are only finitely
many such pairings, so we may pass to a further subsequence of the
$V_i$ where $\partial D_0$ and $\partial D_1$ compress in all of the
$V_i$, and with the same pairing.  It follows that these $V_i$ share a
common non-trivial compression body $W_1$, containing $\partial D_0$
and $\partial D_1$.

We now iterate the argument, finding simultaneous compressions
$\partial D_1, \partial D_2, \partial D_3, \ldots$ for a descending
chain of subsequences of $\{V_i\}$.  This gives rise to an ascending
chain of compression bodies
$W_1 \subset W_2 \subset W_3 \subset \ldots$.  However, an ascending
chain of compression bodies must stabilize after finitely many steps.

If there is some $m$ so that $W_m$ is a handlebody, then infinitely
many of the $V_i$ are equivalent.  If there is some $m$ so that $W_m =
W_n$ for all $n > m$ then the compression body $W = W_m$ is contained
in infinitely many of the $V_i$.  In either case we have completed the
proof of Theorem \ref{Thm:Stability}.
\end{proof}

\section{Infinite diameter} \label{section:diameter}

We will use the following result of Klarreich \cite[Theorem
1.3]{klarreich}.  See also Hamenst\"adt \cite{ham}.

\begin{theorem}\cite[Theorem 1.3]{klarreich}\label{theorem:boundary}
The Gromov boundary of the complex of curves $\C(S)$ is homeomorphic
to the space of minimal foliations on $S$.
\end{theorem}

We may now complete the proof of Theorem \ref{Thm:InfDiam}.

\begin{proof}[Proof of Theorem \ref{Thm:InfDiam}]
Fix a geodesic ray $\gamma \from \N \to \C(S)$, and set
$\gamma_i = \gamma(i)$. Applying Theorem \ref{theorem:boundary}, let
$\lambda$ be the ending lamination associated to $\gamma$.  That is,
after fixing a hyperbolic metric on $S$, and after replacing all
curves by their geodesic representatives, the lamination $\lambda$ is
obtained from any Hausdorff limit of the $\gamma_i$ by deleting
isolated leaves. We say that the $\gamma_i$ \emph{superconverge} to
$\lambda$.  By Proposition \ref{prop:lamination} we may assume that
$\lambda$ does not lie in the zero set of any compression body.

Suppose that $\pi_{\mathcal{H}(S)} \circ \gamma$ has finite
diameter. By Theorem \ref{Thm:Ray} there is a compression body $V$, a
constant $k$, and a sequence of meridians $\partial D_i \in \calD(V)$
so that $d_S(\gamma_i, \partial D_i) \leq k$.  Kobayashi's Lemma
\cite{kobayashi}*{Proposition 2.2} implies the $\partial D_i$ also
superconverge to $\lambda$. It follows that any accumulation point of
the $\partial D_i$, taken in $\PML(S)$, is supported on $\lambda$. So,
picking any measure of full support for $\lambda$, realizes $\lambda$
as an element of $Z(V)$, contradicting our initial choice of
$\lambda$.
\end{proof}

\section{Compression bodies of restricted topological type}
\label{section:restricted}

Consider the poset $\CB(S)$ of compression
bodies, ordered by inclusion, as described above.  We can take the
quotient by the action of the mapping class group of the upper
boundary surface $S$.  This gives the poset of topological types of
compression bodies, which we shall denote $\CB_\top(S)$.  Let $A$ be a
non-empty subset of $\CB_\top(S)$.  We say that $A$ is
\emph{downwardly closed} if $A$ is closed under passing to
subcompression bodies. Define the \emph{restricted compression body
  graph} $\CB(S, A)$ to be the subposet of $\CB(S)$ where we require
all vertices to have topological type lying in $A$.

\begin{theorem}
\label{Thm:restrict}
Let $A$ be a downwardly closed connected non-empty subset of
topological types of compression body. Then the restricted compression
body graph $\mathcal{H}(S, A)$ is an infinite diameter Gromov
hyperbolic metric space.
\end{theorem}

\noindent Our methods apply in this generality, but in fact Theorem
\ref{Thm:restrict} also follows quickly from Theorem
\ref{Thm:InfDiam}.

\begin{proof}[Proof of Theorem \ref{Thm:restrict}]
If $A \subset B$ are downwardly closed subsets of $\CB_\top(S)$ then
the inclusion $\CB(S,A) \subset \CB(S,B)$ is simplicial and is
coarsely onto.  As a special case, $\CB(S, A) \subset \calH(S)$ is
simplicial and coarsely onto.  Thus $\CB(S, A)$ has infinite diameter.
\end{proof}


\section{Loxodromics} \label{section:loxodromics}

Let $G$ be a group acting by isometries on a Gromov hyperbolic space
$(X, d_X)$ with basepoint $x_0$.  We say that an element $h \in G$
acts \emph{loxodromically} if the translation length
\[ \tau(h) = \lim_{n \to \infty} \tfrac{1}{n} d_X( x_0, h^n x_0) \]
is positive.  This implies that $h$ has a unique pair of fixed points
$\{\lambda^+_h, \lambda^-_h\}$ in the Gromov boundary $\partial X$.
Masur and Minsky \cite{mm1} showed that for the action of the mapping
class group on the curve complex, an element is loxodromic if and only
if it is pseudo-Anosov.  We say two loxodromic isometries $h_1$ and
$h_2$ are \emph{independent} if their pairs of fixed points
$\{ \lambda^+_{h_1}, \lambda^-_{h_1} \}$ and
$\{ \lambda^+_{h_2}, \lambda^-_{h_2} \}$ in the Gromov boundary
$\partial X$ are disjoint.

We remark that there are many pseudo-Anosov elements for which no
power extends over any compression body.  For example, if a power of
$g$ extends over a compression body, then some power of $g$ preserves
a rational subspace of first homology.  However, generic elements of
$\operatorname{Sp}(2g, \Z)$ do not do this, see for example
\cites{dt,rivin}.  We now give an alternate geometric argument to show
the existence of pseudo-Anosov elements for which no power extends over
a compression body.

\begin{lemma}\label{lemma:pA}
Let $S$ be a closed surface of genus at least two.  Then there is a
mapping class group element $h \in \text{Mod}(S)$ which acts
loxodromically on the compression body graph $\calH(S)$.
\end{lemma}


We prove the following, more general result.  Let $G$ act by
isometries on $X$, and let $Y$ be a quasi-convex subset of $X$.  We
say that $G$ \emph{acts loxodromically} on $(X, Y)$, if $G$ contains a
loxodromic element $g$ whose quasi-axis is contained in a bounded
neighbourhood of $Y$.

\begin{lemma}\label{lemma:coarse}
Let $G$ be a group acting by isometries on a Gromov hyperbolic space
$X$, and let $\calY$ be a collection of uniformly quasi-convex subsets
of $X$.  Furthermore, suppose that $G$ acts with unbounded orbits on
$X_\calY$, and there is a quasi-convex set $Y \in \calY$, such that
$G$ acts loxodromically on $(X, Y)$.  Then $G$ contains an element
which acts loxodromically on $X_\calY$.
\end{lemma}

We now show that Lemma \ref{lemma:coarse} implies Lemma
\ref{lemma:pA}.

\begin{proof}[Proof of Lemma \ref{lemma:pA}]
Let $G = \text{Mod}(S)$ be the mapping class group acting on
$X = \calC(S)$, the complex of curves, which is Gromov hyperbolic.
Let $\calY$ consist of the collection of disc sets of compression
bodies, which is a collection of uniformly quasi-convex subsets of
$X$.  The space $X_\calY = \calH(S)$ has infinite diameter by Theorem
\ref{Thm:InfDiam}.  The mapping class group acts coarsely transitively
on $\calH(S)$, and so in particular acts with unbounded orbits.

Let $V$ be a handlebody and pick a pair of discs $D$ and $D'$ whose
boundaries $\alpha$ and $\alpha'$ fill $S$.  We define $f$ to be the
product of a right Dehn twist on $\alpha$, followed by a left Dehn
twist on $\alpha'$.
The mapping class group element $f$ is pseudo-Anosov by work of
Thurston \cite{thurston}.  The element $f$ extends over the
compression body $V$, and the images of $\alpha$ under powers of $f$
is a quasi-axis for $f$, which is contained in $\calD(V)$, so $f$ acts
loxodromically on $\calD(V)$.

Lemma \ref{lemma:coarse} then implies that $G$ contains an element
which acts loxodromically on $\calH(S)$.
\end{proof}

The proof we present relies on the fact that $(K, c)$-quasi-geodesics
in the curve complex $\C(S)$ project to reparameterized
$(K', c')$-quasi-geodesics in the compression body graph $\calH(S)$,
where $K'$ and $c'$ depend only on $K$ and $c$.

We will use the following properties of coarse negative curvature.  We
omit the proofs of Propositions \ref{prop:npp diam}, \ref{prop:gromov
  product} and \ref{prop:qg gromov} below, as they are elementary
exercises in coarse geometry.

We say a path $\gamma$ is a \emph{$(K, c, L)$-local quasi-geodesic},
if every subpath of $\gamma$ of length $L$ is a
$(K, c)$-quasi-geodesic.  The following theorem gives a classical
``local to global'' property for negative curvature.

\begin{theorem}\cite{cdp}*{Chapter 3, Th\'eor\`eme 1.4, page 25}
\label{theorem:localqg}
Given constants $\delta, K$ and $c$, there are constants $K', c'$ and
$L_0$, such that for any $L \ge L_0$, any $(K, c , L)$-local
quasi-geodesic in a $\delta$-hyperbolic space $X$ is a
$(K', c')$-quasi-geodesic.
\end{theorem}

In fact, we will make use of the following special case of this
result.  A \emph{piecewise geodesic} is a path $\gamma$ which is a
concatenation of geodesic segments $\gamma_i = [x_i, x_{i+1}]$.  We
say a piecewise geodesic $\gamma$ has \emph{$R$-bounded Gromov
  products} if $\gp{x_i}{x_{i-1}}{x_{i+1}} \le R$ for all $i$.

\begin{proposition}\label{prop:bounded gp}
Given constants $\delta$ and $R$, there is are constants $L, K$ and
$c$ such that if $\gamma$ is a piecewise geodesic in a
$\delta$-hyperbolic space, with $R$-bounded Gromov products, and
$\norm{\gamma_i} \ge L$ for all $i$, then $\gamma$ is a
$(K, c)$-quasi-geodesic.
\end{proposition}

\begin{proof}
Consider a subpath of $\gamma$ which is concatenation of two geodesic
segments $[x_{i - 1}, x_i]$ and $[x_i, x_{i + 1}]$.  Then by the
definition of the Gromov product, this subpath is a
$(1, 2 \gp{x_i}{x_{i-1}}{x_{i+1}} )$-quasi-geodesic: that is, a
$(1, 2R)$-quasi-geodesic. As each subpath of $\gamma$ of length at
most $L$ is contained in at most $2$ geodesic subsegments, $\gamma$ is
therefore a $(1, 2R, L)$-local quasi-geodesic.  By Theorem
\ref{theorem:localqg}, given $\delta, 1$ and $2R$, there are constants
$L, K$ and $c$, such that any $(1, 2R, L)$-local quasi-geodesic is a
$(K, c)$-quasi-geodesic, as required.
\end{proof}

We now record the following property of quasiconvex sets inside a
hyperbolic space.  This follows from thin triangles and we omit the
proof.

\begin{proposition}\label{prop:npp diam}
Given constants $\delta$ and $Q$, there is a constant $R_0$, such that
any constant $R \ge R_0$ has the following properties: let $Y$ and
$Y'$ be $Q$-quasiconvex sets in a $\delta$-hyperbolic space $X$.  If
the distance between $Y$ and $Y'$ is at least $R$, then the nearest
point projection of $Y$ to $Y'$ has diameter at most $R$.
Furthermore, if $Y''$ is a $Q$-quasiconvex set distance at least $R$
from $Y$, and distance at most $R$ from $Y'$, then the distance
between the nearest point projections of $Y'$ and $Y''$ to $Y$ is at
most $R$.
\end{proposition}

Let $\calY$ be a collection of $Q$-quasi-convex sets in a
$\delta$-hyperbolic space $X$, such that $X_\calY$ has infinite
diameter.

Let $\calZ = \{ Z_i \}_{i \in N}$ be an ordered subcollection of the
$\calY$, where $N$ is a set of consecutive integers in $\Z$.  If
$N = \Z$, then we say that $\calZ$ is \emph{bi-infinite}.  We say that
$\calZ$ is \emph{$L$-well-separated} if $d_X(Z_i, Z_{i+1}) \ge L$ for
all $i$, and furthermore the distance between the nearest point
projections of $Z_{i-1}$ and $Z_{i+1}$ to $Z_i$ is also at least $L$.
Then $\calZ$ determines a collection of piecewise geodesics as
follows.  For each $Z_i$, let $p_i$ be a point in the nearest point
projection of $Z_{i-1}$ to $Z_i$, and let $q_i$ be a point in the
nearest point projection of $Z_{i+1}$ to $Z_i$.  Let $\gamma_\calZ$ be
a path formed from the concatenation of the geodesic segments
$[p_i, q_i]$ and $[q_i, p_{i+1}]$.  We will call such a path
$\gamma_\calZ$ a \emph{$\calZ$-piecewise geodesic}.

\begin{proposition}\label{prop:gromov product}
Given constants $\delta$ and $Q$, there is a constant $R$, such that
for any $Q$-quasi-convex set $Z$ in a $\delta$-hyperbolic space $X$,
and for any points $x$ in $X$ and $z$ in $Z$, with $p$ a nearest point
in $Z$ to $x$, then the Gromov product $\gp{p}{x}{z} \le R$, where $R$
depends only on $\delta$ and $Q$.
\end{proposition}

We now show that an $L$-well separated collection $\calZ$ of ordered
$Q$-quasi-convex sets gives rise to a natural family of quasi-geodesics.

\begin{proposition}\label{prop:well-separated}
Given constants $\delta$ and $Q$, there are constants $K, c$ and
$L_0$, such that for any $L \ge L_0$, and any collection $\calZ$ of
$L$-well-separated ordered $Q$-quasi-convex sets in a
$\delta$-hyperbolic space $X$, any $\calZ$-piecewise geodesic is a
$(K, c)$-quasi-geodesic.
\end{proposition}

\begin{proof}
By Proposition \ref{prop:npp diam}, there is a constant $R_1$, which
only depends on $\delta$ and $Q$, such that if $Z$ and $Z'$ are two
$Q$-quasi-convex sets distance at least $R_1$ apart, then the nearest
point projection of $Z$ to $Z'$ has diameter at most $R_1$.

By Proposition \ref{prop:gromov product}, for any $x \in X$ and $z \in
Z$, with $p$ a nearest point in $Z$ to $x$, then $\gp{p}{x}{z} \le
R_2$.  For two adjacent segments $[p_i, q_i], [q_i, p_{i+1}]$ in a
$\calZ$-piecewise geodesic, $q_i$ need not be the closest point on
$Z_i$ to $p_i$, but by Proposition \ref{prop:npp diam}, it is distance
at most $R_1$ from the nearest point $q'_i$ on $Z_i$ to $p_i$.  By the
definition of the Gromov product, if $d_X(q_i, q'_i) \le R_1$, then
the difference between the Gromov products $\gp{q_i}{p_i}{q_{i+1}}$
and $\gp{q'_i}{p_i}{q_{i+1}}$ is at most $R_1$, and so any
$\calZ$-piecewise geodesic has $(R_1+R_2)$-bounded Gromov products.
So, by Proposition \ref{prop:bounded gp}, there are constants $L, K$
and $c$, such that if every segment of a $\calZ$-piecewise geodesic
has length at least $L$, it is a $(K, c)$-quasi-geodesic.
\end{proof}

We say that an ordered set $\calZ = \{ Z_i\}_{i \in \Z}$ of uniformly
quasi-convex subsets of $X$ is \emph{$(L, M)$-$\calY$-separated}, if
$\calZ$ is $L$-well-separated in $X$,
and $d_{X_\calY}(\pi(Z_i), \pi(Z_{i+1})) \ge M$ for all $i$.

\begin{proposition}\label{prop:qg gromov}
Given constants $\delta, K$ and $c$ there is a constant $R$, such that
for any reparameterized $(K, c)$-quasi-geodesic
$\gamma \colon \R \to X$ in a $\delta$-hyperbolic space $X$, for any
three numbers $r \le s \le t$, the Gromov product
$\gp{\gamma(s)}{\gamma(r)}{\gamma(t)} \le R$.
\end{proposition}

\begin{proposition}\label{prop:bi-infinite}
Given constants $\delta$ and $Q$, there are constants $K, c, L_0$ and
$M_0$, such that for any collection $\calY$ of $Q$-quasi-convex sets,
for any $L \ge L_0$ and $M \ge M_0$, and any bi-infinite collection
$\calZ$ of $(L, M )$-$\calY$-separated sets, then the image of any
$\calZ$-piecewise geodesic in $X_\calY$ is a bi-infinite
reparameterized $(K, c)$-quasi-geodesic.
\end{proposition}

\begin{proof}
By Proposition \ref{prop:well-separated}, given constants $\delta_1$
and $Q$, there are constants $K_1, c_1$ and $L_1$ such that for any
collection $\calZ$ of $L$-well-separated $Q$-quasi-convex sets, with
$L \ge L_1$, any $\calZ$-piecewise geodesic $\gamma$ is a
$(K_1, c_1)$-quasi-geodesic in $X$.  By Theorem \ref{theorem:quotient}
there are constants $\delta_2, K_2$ and $c_2$ such that the image of
$\gamma$ in $X_\calY$ is a reparameterized $(K_2, c_2)$-quasi-geodesic
in the $\delta_2$-hyperbolic space $X_\calY$.

By Proposition \ref{prop:qg gromov}, given constants $\delta_2, K_2$
and $c_2$, there is a constant $R$, such that for any $i \le j \le k$,
there is a bound on their Gromov product in $X_\calY$: that is,
$\gp{\pi(p_j)}{\pi(p_i)}{\pi(p_k)}^{X_\calY} \le R$.  Therefore, by
Proposition \ref{prop:bounded gp}, given constants $\delta_2$ and $R$,
there are constants $K_2, c_2$ and $M$ such that as long as
$d_{X_\calY}(\pi(p_i), \pi(p_{i+1})) \ge M$, the piecewise geodesic in
$X_\calY$ formed from geodesic segments $[\pi(p_{i}), \pi(p_{i+1})]$
is a $(K_2, c_2)$-quasi-geodesic, and in particular is bi-infinite.  So
for any collection $\calZ$ of $(L, M)$-$\calY$-separated sets, the
image of any $\calZ$-piecewise geodesic in $X_\calY$ is a
reparameterized bi-infinite $(K_2, c_2)$-quasi-geodesic.
\end{proof}

We may now complete the proof of Lemma \ref{lemma:coarse}.

\begin{proof}[Proof (of Lemma \ref{lemma:coarse}).]
Let $\delta$ be the constant of hyperbolicity for $X$, and let $Q$ be
a constant such that all sets $Y \in \calY$ are $Q$-quasiconvex.  Let
$f$ be an element of $G$ which acts loxodromically on $Y \in \calY$.
The isometry $f$ acts elliptically on $X_\calY$, coarsely fixing $Y$.
In particular, there is a constant $A$, depending only on $\delta$ and
$Q$, such that $d_{X_\calY}(Y, f^\ell Y) \le A$ for all $\ell \in \Z$.

Given $\delta$ and $Q$, let $R_0$ be the constant from Proposition
\ref{prop:npp diam}, and choose $R \ge R_0 + A$.  In particular, for
any two sets $Y$ and $Y'$ in $\calY$, distance at least $R$ apart in
$X$, the nearest point projection of $Y'$ to $Y$ has diameter at most
$R$, and if $Y''$ is another set in $\calY$, distance at least $R$
from $Y$ and at most $R$ from $Y'$, then the nearest point projections
of $Y'$ and $Y''$ to $Y$ are distance at most $R$ apart.

Consider a pair of numbers $L$ and $M$ with $M \ge L \ge R$.  The
group $G$ acts with unbounded orbits on $X_\calY$, so for any such
number $M$, there is a group element $g \in G$ such that
$d_{X_\calY}(Y, g Y) \ge M + A$.  This implies that for all $\ell$,
$d_X(Y, g Y) \ge M + A$, and equivalently, for all $\ell$,
$ d_X(Y, g^{-1} Y) \ge M + A$.  As
$d_{X_\calY}(Y, f^\ell g Y) = d_{X_\calY}(f^{- \ell}Y, g Y)$ this
implies that for all $\ell$, $d_{X_\calY}(Y, f^\ell g Y) \ge M$, and
equivalently, that for all $\ell$,
$d_{X_\calY}(Y, g^{-1} f^{-\ell} Y) \ge M$.  As the map from $X$ to
$X_\calY$ is $1$-Lipschitz, this implies that
$d_{X}(Y, f^\ell g Y) \ge M$, and equivalently
$d_{X}(Y, g^{-1} f^{-\ell} Y) \ge M$.

As we have chosen $M \ge R$, the nearest point projection of
$f^\ell g Y$ to $Y$ has diameter at most $R$.  Similarly, the nearest
point projection of $g^{-1} f^{-\ell} Y$ to $Y$ has diameter at most
$R$.  Furthermore, for all $\ell$ and $\ell'$,
$d_X( g^{-1} f^{-\ell} Y, g^{-1} f^{-\ell'} Y ) \le A$, and so for all
$\ell$ and $\ell'$, the nearest point projections of
$g^{-1} f^{-\ell} Y$ and $g^{-1} f^{-\ell'} Y$ to $Y$ are distance at
most $R$ apart.  Therefore, as $f$ acts loxodromically on $Y$, for any
number $L \ge R$ there is a sufficiently large number $\ell$ such that
$\pi_Y(g^{-1} f^{-\ell} Y)$ and $\pi_Y(f^\ell g Y)$ are distance at
least $L$ apart.  This is illustrated schematically in Figure
\ref{pic:npp Y}.

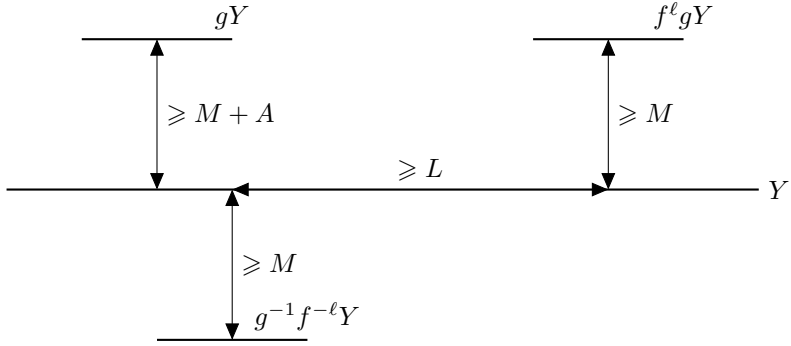
\begin{figure}[H]
\begin{center}
\begin{tikzpicture}

\draw [thick] (0, 0) -- (10, 0) node [right] {$Y$};

\draw [triangle 45-triangle 45] (2, 0) -- node [midway, right]
{$\ge M + A$} (2, 2);
\draw [thick] (1, 2) -- (3, 2) node [above] {$g Y$};

\draw [triangle 45-triangle 45] (3, 0) -- node [midway, right] {$\ge M$} (3, -2);
\draw [thick] (2, -2) -- (4, -2) node [above] {$g^{-1} f^{-\ell} Y$};

\draw [triangle 45-triangle 45] (8, 0) -- node [midway, right] {$\ge M$} (8, 2);
\draw [thick] (7, 2) -- (9, 2) node [above]
{$f^\ell g Y$};

\draw [triangle 45-triangle 45] (3, 0) -- node [midway, above] {$\ge
  L$} (8, 0);

\end{tikzpicture}
\end{center}
\caption{Nearest point projections of $g^{-1} Y$ and $f^\ell g Y$ to
  $Y$ in $X$.}
\label{pic:npp Y}
\end{figure}

Set $h = f^\ell g$.  Then by Proposition \ref{prop:bi-infinite}, for
$L$ and $M$ sufficiently large, the set
$\calZ = \{ h^n Y\}_{n \in \Z}$ is a bi-infinite collection of
$(L, M)$-$\calY$-separated sets, with the property that any
$\calZ$-piecewise geodesic quasi-axis for $h$ projects to a
bi-infinite quasi-axis for $h$ in $X_\calY$, and so $h$ acts
loxodromically on $X_\calY$, as required.
\end{proof}

The following corollary implies that every subgroup in the Johnson
filtration contains an element which acts loxodromically on $\calH$.

\begin{corollary}\label{corollary:johnson}
Let $X$ be a $\delta$-hyperbolic space, and let $\calY$ be a
collection of uniformly quasi-convex sets, such that the
electrification $X_\calY$ has infinite diameter.  Let $G$ be a group
which acts on $X$ by isometries, and which contains an element which
acts loxodromically on $X_\calY$.  Let $H$ be a normal subgroup of
$G$, which contains an element which acts loxodromically on $X$.  Then
$H$ contains an element which acts loxodromic on $X_\calY$.
\end{corollary}

We will use the following property of independent loxodromics.

\begin{proposition}\label{prop:schottky}
Let $g$ and $h$ be independent loxodromics on a $\delta$-hyperbolic
space $X$.  Then there is a constant $m$, such that $g^m$ and $h^m$
freely generate a free group, all of whose non-trivial elements act
loxodromically on $X$.  Furthermore, the orbit map applied to this
free group is a quasi-isometric embedding.
\end{proposition}

\begin{proof}[Proof (of Corollary \ref{corollary:johnson})]
Suppose that $g \in G$ acts loxodromically on $X_\calY$, and $h \in H$
acts loxodromically on $X$.  If $h$ acts loxodromically on $X_\calY$,
then we are done.  If not, then $g$ and $h$ are independent as
loxodromics acting on $X$.  We may choose $\{ h^n x_0 \}_{n \in \Z}$
as a quasi-axis $\gamma_h$ for $h$.  Note that the axis is
quasi-convex in $X$.

Given constants $L$ and $M$, we will show that there are positive
integers $l$ and $m$, such that if $f = g^l h^m g^{-l} h^m$, the
ordered collection of sets
$\calZ = \{ Z_n = f^n \gamma_h \}_{n \in \N}$ is a bi-infinite
collection of uniformly quasi-convex $(L, M)$-$\calY$-separated sets
in $X$.  As $\calZ$ consists of translates of a single quasi-convex
set, the $Z_i$ are uniformly quasi-convex.

For $l$ sufficiently large, the nearest point projection of
$f \gamma_h = g^l h^m g^{-l} \gamma_h$ to $\gamma_h$ lies in a bounded
neighbourhood of some vertex of $\gamma_h$, say $x_0$.  Similarly, for
$l$ and $m$ sufficiently large, the nearest point projection of
$f^{-1} \gamma_h$ lies in a bounded neighbourhood of $h^{-m} x_0$.  By
choosing $m$ sufficiently large, this implies that
\[ d_X( \pi_{\gamma_h} ( f \gamma_h ), \pi_{\gamma_h}( f^{-1}
\gamma_h) ) \ge L. \]
Then as $\calZ$ is $f$-invariant, it is $L$-well-separated in $X$.

As $g$ and $h$ are independent, the nearest point projection of
$\gamma_h$ to $\gamma_g$ has bounded diameter in $X$, so the nearest
point projection of $\pi(\gamma_h)$ to $\pi(\gamma_g)$ also has
bounded diameter in $X_\calY$.  As $g$ acts loxodromically on
$X_\calY$, for any constant $M$ there is a positive integer $l$ such
that $d_{X_\calY}(\gamma_h, f \gamma_h) \ge M$, and again as $\calZ$
is $f$-invariant, this implies that $\calZ$ is $M$-$\calY$-separated.
Proposition \ref{prop:bi-infinite} then implies that $f$ acts
loxodromically on $X_\calY$.
\end{proof}

We now show that if two pseudo-Anosov elements $f$ and $g$ act
loxodromically on the compression body graph $\calH(S)$, and act
independently on $\C(S)$, then they also act independently on
$\calH(S)$.

\begin{proposition}
Let $X$ be a $\delta$-hyperbolic space, and let $\calY$ be a
collection of uniformly quasi-convex subsets of $X$.  Let $f$ and $g$
act loxodromically on $X_\calY$, and let them be independent as
loxodromics acting on $X$.  Then they are independent loxodromics
acting on $X_\calY$.
\end{proposition}

\begin{proof}
Let $\lambda^\pm_f$ be the fixed points of $f$ in $\partial X$, and
let $\lambda^\pm_g$ be the fixed points of $g$ in $\partial X$. As $f$
and $g$ act independently on $X$, all of these points are distinct.
Let $\gamma$ be a geodesic from $\lambda^+_f$ to $\lambda^+_g$. The
image of $\gamma$ under projection to $X_\calY$ is a reparameterized
geodesic connecting the fixed points of $f$ and $g$ in
$\partial X_\calY$ with at least one point in the interior of
$X_\calY$.  Thus $\pi \circ \gamma$ is a bi-infinite geodesic
connecting the fixed points, so they are distinct.
\end{proof}

\section{Applications} \label{section:applications}

As an application of our methods, we give an alternative proof of a
result of Biringer, Johnson and Minsky \cite{bjm}*{Theorem 1.1}, and a
slightly stronger version of a result of Lubotzky, Maher and Wu
\cite{lmw}*{Theorem 2}, using results of Maher and Tiozzo
\cite{maher-tiozzo}.

\begin{theorem}\label{T:bjm}
Let $g$ be a pseudo-Anosov element of the mapping class group of a
closed orientable surface. Then the stable lamination $\lambda^+_g$ of
$g$ is contained in the limit set for a compression body $V$ if and
only if the unstable lamination $\lambda^-_g$ is contained in the
limit set for $V$, and furthermore, there is a compression body
$V' \subseteq V$ such that some power of $g$ extends over $V'$.
\end{theorem}

\begin{proof}
Suppose the stable lamination $\lambda^+_g$ for the pseudo-Anosov
element $g$ lies in the disc set of a compression body $V_1$.  Let
$\lambda^-_g$ be the unstable lamination for $g$, and let $\gamma$ be
a geodesic axis for $g$ in $\mathcal{C}(S)$.  Choose a basepoint
$\gamma(0)$ on $\gamma$, and assume that $\gamma$ is parameterized by
distance from $\gamma_0$ in $\mathcal{C}(S)$, with
$\lim_{n \to \infty} \gamma(n) = \lambda^+_g$.  Consider the sequence
of compression bodies $W_i = g^{-i} V_1$.  Possibly after re-indexing
the $W_i$, we may assume that the geodesic from $\gamma(0)$ to
$\gamma(-i)$ is contained in $k$-neighbourhood of $\mathcal{D}(W_i)$,
where $k$ depends only on the constant of hyperbolicity and the
quasi-convexity constants for the disc sets.  We may therefore apply
Theorem \ref{Thm:Stability} above, which then implies that there is a
compression body $V_2$, contained in infinitely many of the $W_i$.
Furthermore, the geodesic ray from $\gamma(0)$ to $\lambda^-(g)$ is
contained in a bounded neighbourhood of $\mathcal{D}(V_2)$.  In
particular, $V_2 \subseteq g^{k_1} V_1$ for some $n_1$, and we may
assume that $k_1 \not = 0$ as $V_2$ is contained in infinitely many
$g^i V_1$.

We may then repeat this procedure using positive powers of $g$, to
produce a compression body $V_3$, contained in infinitely many
$\mathcal{D}( g^i V_2 )$, such that $\lambda_g^+$ is contained in the
limit set of $\mathcal{D}(V_3)$.  Iterating this procedure produces a
descending chain of compression bodies
$V_{n+1} \subseteq g^{k_n} V_{n}$, which has the property that the
$V_i$ have limit sets containing $\lambda^+_g$ if $i$ is odd, and
$\lambda^-_g$ if $i$ is even.  As before, we may assume that the
integers $k_n$ are not zero.  A descending chain of compression bodies
must eventually stabilize with $V_n = g^{k_n} V_{n+1}$ for some $n$
and $k_n \not = 0$.  Therefore, there is a single compression body
$V = V_n$, which contains both $\lambda^+_g$ and $\lambda_g^-$, as
required.  Finally, we observe that as $k_n \not = 0$, some power of
$g$ extends over this compression body.
\end{proof}

Lubotzky, Maher and Wu showed that a random walk on the mapping class
group gives a hyperbolic manifold with a probability that tends to one
exponentially quickly, assuming that the probability distribution
$\mu$ generating the random walk is \emph{complete}: that is, the
limit set of the subgroup generated by the support of $\mu$ is dense
in Thurston's boundary for the mapping class group, $\mathcal{PML}$.
As the compression body graph is an infinite diameter Gromov
hyperbolic space, we may apply the results of Maher and Tiozzo
\cite{maher-tiozzo}, to replace this hypothesis with the assumption
that the support of $\mu$ contains a pair of independent loxodromic
elements for the action of the subgroup on the compression body graph.
We shall write $w_n$ for a random walk of length $n$ on the mapping
class group of a closed orientable surface of genus $g$, generated by
a probability distribution $\mu$, and $M(w_n)$ for the corresponding
\emph{random Heegaard splitting}: that is the $3$-manifold obtained by
using the resulting mapping class group element $w_n$ as the gluing
map for a Heegaard splitting.

\begin{proposition}
Every isometry of the compression body graph $\mathcal{H}$ is either
elliptic or loxodromic.
\end{proposition}

\begin{proof}
Biringer and Vlamis \cite{bv}*{Theorem 1.1} showed that the isometry
group of $\mathcal{H}$ is equal to the mapping class group.  If $g$ is
not pseudo-Anosov, then $g$ acts elliptically on the curve complex,
and hence acts elliptically on $\mathcal{H}$.
Let $g$ be a pseudo-Anosov element of the mapping class group.
Let $\gamma$ be an axis for $g$ in the curve complex $\C(S)$, and let
$\gamma(0)$ be a choice of basepoint for $\gamma$.

Suppose that $g$ does not act loxodromically, it follows that there is
a sequence of compression bodies $V_i$, whose nearest point
projections to $\gamma$ have diameters tending to infinity.  Recall
that the disc sets $\mathcal{D}(V_i)$ are uniformly quasi-convex.
Thus there is a constant $k$, depending on the quasi-convexity
constants of the disc sets and the hyperbolicity constant of the curve
complex $\C(S)$, such that (after passing to a subsequence and
re-indexing) the intersection of $\mathcal{D}(V_i)$ with a
$k$-neighbourhood of $\gamma$ has diameter at least $i$.  By
translating the disc sets by powers of $g$, and possibly passing to a
further subsequence and re-indexing, we may assume that both
$\gamma(0)$ and $\gamma(i)$ are distance at most $k$ from
$\mathcal{D}(g^{n_i} V_i)$.  We may further relabel the disc sets and
just write $V_i$ for the given translate $g^{n_i} V_i$.

We may now apply our stability result, Theorem \ref{Thm:Stability}.
This implies that there is a single compression body $W$ contained in
infinitely many of the $V_i$, and so in particular the positive limit
point $\lambda^+_g$ of $g$ is contained in the limit set of
$\mathcal{D}(W)$.  Thus there is a compression body $W'$ such that
some power of $g$ extends over $W'$.  Thus $g$ acts elliptically, and
we are done.
\end{proof}

\begin{theorem}
Let $\mu$ be a probability distribution on the mapping class group of
a closed orientable surface of genus $g$, whose support has bounded
image in the compression body graph $\mathcal{H}$, and which contains
two independent pseudo-Anosov elements whose stable laminations do not
lie in the limit set of any compression body.  Then the probability
that $M(w_n)$ is hyperbolic and of Heegaard genus $g$ tends to one
exponentially quickly.
\end{theorem}

\begin{proof}
Maher and Tiozzo \cite{maher-tiozzo}*{Theorem 1.2} show that given a
countable group $G$ acting non-elementarily on a separable Gromov
hyperbolic space $(X, d_X)$, a finitely supported random walk on $G$
has positive drift with exponential decay: that is, there are
constants $L > 0, K \ge 0$, and $c < 1$ such that
\[ \P \left[ d_X(x_0, w_n x_0) \ge Ln \right] \ge 1 - Kc^n.  \]
We may apply this result to the mapping class group $G$ acting on the
compression body graph $\mathcal{H}$. A subgroup of $G$ acts
\emph{non-elementarily} on $\mathcal{H}$ if $G$ contains two
independent loxodromic elements.  Geodesics in $\C(S)$ project to
reparameterized quasi-geodesics in $\mathcal{H}$, so by Theorem
\ref{Thm:Ray}, if the stable and unstable laminations of a
pseudo-Anosov element $h$ do not lie in the limit set of some
compression body, then the image of the axis of $h$ in $\C_\calD(S)$
is a bi-infinite quasi-axis for the action of $h$ on the compression
body graph $\mathcal{H}$, and so in particular $h$ acts loxodromically
on $\mathcal{H}$.  If $h_1$ and $h_2$ are two pseudo-Anosov elements,
which act loxodromically on the compression body graph $\mathcal{H}$,
and act independently on the curve complex $\C(S)$, then in fact they
act independently on the compression body graph $\mathcal{H}$, as the
geodesic from $\lambda^+_{h_1}$ to $\lambda^+_{h_2}$ in the curve
complex $\C(S)$ projects to an reparameterized quasi-geodesic in
$\mathcal{H}$, which has infinite diameter by Theorem \ref{Thm:Ray}.
Finally, we observe that the compression body graph $\mathcal{H}$ is a
countable simplicial complex, and so is separable: that is, has a
countable dense subset, and so the hypotheses of \cite[Theorem
1.2]{maher-tiozzo} are satisfied.

The \emph{Hempel distance} of a Heegaard splitting $M(h)$ is the
distance in the curve complex between the disc sets of the two
handlebodies.  This is bounded below by distance in the compression
body graph $d_{\mathcal{H}}(x_0, w_n x_0)$.  Since distance in the
compression body graph has positive drift with exponential decay, so
does Hempel distance.  Finally, we observe that if the Hempel distance
is at least three, then $M(w_n)$ is hyperbolic, by work of Hempel
\cite{hempel}*{Corollaries 3.7 and 3.8}, and Perelman's proof of
Thurston's geometrization conjecture \cite{morgan-tian}.  If the
splitting distance is greater than $2g$, then the given Heegaard
splitting is a minimal genus Heegaard splitting for $M(w_n)$, by work
of Scharlemann and Tomova \cite{st}*{page 594}.
\end{proof}

\appendix
\section{Train tracks}

In this section we review the work of Masur and Minsky \cite{mm3}.
They prove the version of Proposition \ref{prop:splitting seq} we
require, but we find it convenient to write down an argument which
differs from theirs in certain details.

Let $a$ and $b$ be two essential simple closed curves in $S$ in
minimal position.  A \emph{subarc} of a simple closed curve is a
closed connected subinterval.  We will only every consider subarcs of
$a$ or $b$ whose endpoints lie in $a \cap b$.  Given a pair of curves
$a$ and $b$ in minimal position, a \emph{bicorn} is a simple closed
curve, denoted by $c = (a_i, b_i)$, consisting of the union of one
subarc $a_i$ of $a$ and one subarc $b_i$ of $b$.

A subarc of $a$ or $b$ is \emph{innermost} with respect to $a \cap b$
if its endpoints lie in $a \cap b$, and its interior is disjoint from
$a \cap b$.  We may abuse notation by referring to these innermost
subarcs as components of either $a \setminus b$ or $b \setminus a$,
though in fact we wish to include their endpoints.  An innermost arc
$b_i$ of $b$ with respect to $a \cap b$ is a \emph{returning arc} if
both endpoints lie on the same side of $a$ in $S$.  This is
illustrated on the left hand side of Figure \ref{fig:arc} below.

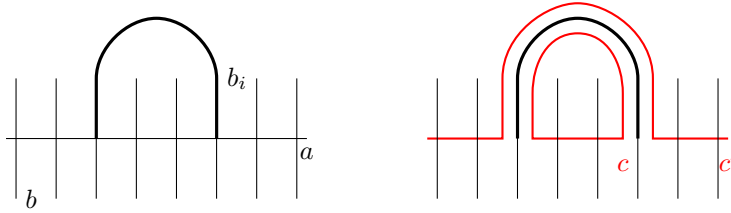
\begin{figure}[H]
\begin{center}
\begin{tikzpicture}[scale=0.4]

\draw (0, 0) -- (10, 0) node [below] {$a$};

\draw (0.333, -2) node [right] {$b$};

\foreach \i in {0.333,1.666,...,10}{
    \draw (\i,-2) -- (\i, 2); 
}

\draw [very thick] (3,0) --
                   (3,2) .. controls (3,3) and (4,4) .. 
                   (5, 4) .. controls (6,4) and (7,3) ..
                   (7, 2) node [right] {$b_i$} --
                   (7, 0);

\begin{scope}[xshift=14cm]
\draw [thick, red] (0, 0) -- 
      (2.5, 0) --
      (2.5, 2) .. controls (2.5, 3.5) and (4, 4.5) ..
      (5, 4.5) .. controls (6, 4.5) and (7.5, 3.5) ..
      (7.5, 2) --
      (7.5, 0) --
      (10, 0) node [below] {$c' \strut$};

\draw [thick, red] (3.5,0) --
      (3.5, 1.5) .. controls (3.5, 2.5) and (4, 3.5) ..
      (5, 3.5) .. controls (6, 3.5) and (6.5, 2.5) ..
      (6.5, 1.5) --
      (6.5, 0) node [below] {$c \strut$} --
      cycle;

\foreach \i in {0.333,1.666,...,10}{
    \draw (\i,-2) -- (\i, 2); 
}

\draw [very thick] (3,0) --
                   (3,2) .. controls (3,3) and (4,4) .. 
                   (5, 4) .. controls (6,4) and (7,3) ..
                   (7, 2) --
                   (7, 0);
\end{scope}

\end{tikzpicture}
\end{center}
\caption{An innermost arc $b_i$ of $b$ forming a returning arc for
  $a$.} \label{fig:arc}
\end{figure}

Given a simple closed curve $a$ and a returning arc $b_i$, we may
produce a new simple closed curve by \emph{arc surgery} of $a$ with
respect to $b_i$.  There are two possible bicorns, $c$ and $c'$,
formed from the union of $b_i$ with one of the two subarcs of $a$ with
endpoints $\partial b_i$.  These are illustrated on the right hand
side of Figure \ref{fig:arc}, where we have isotoped the replacement
curves to be disjoint from $b_i$ for clarity.

We say a bicorn is \emph{returning} if $b_i$ is a returning arc for
$a$.  We say a sequence of returning bicorns $(a_i, b_i)$ is
\emph{nested} if $a_{i+1} \subset a_{i}$, and $b_{i+1}$ is a returning
arc for $c_i$, for all $i$.  An adjacent pair in a sequence of
returning bigons is illustrated in Figure \ref{pic:nested bicorns}.

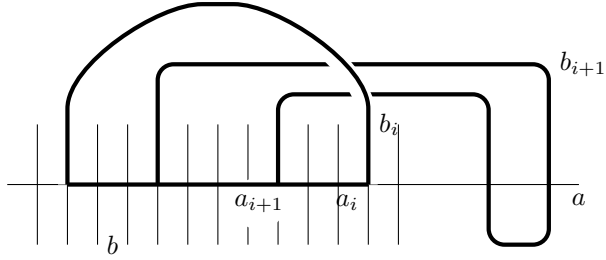
\begin{figure}[H]
\begin{center}
\begin{tikzpicture}[scale=0.4]

\draw (-2, 0) -- (17, 0) node [below] {$a$};

\draw (1, -2) node [right] {$b$};

\foreach \i in {-1,0,...,11}{
    \draw (\i,-2) -- (\i, 2); 
}

\draw [ultra thick, rounded corners=6pt] (3,0) --
                    (3, 4) --
                    (16, 4) node [right] {$b_{i+1}$} --
                    (16, -2) --
                    (14, -2) --
                    (14, 3) --
                    (7,3) --
                    (7, 0);

\draw [line width=0.25cm, white] (0,0) --
                   (0,2) .. controls (0,4) and (3,6) .. 
                   (5, 6) .. controls (7,6) and (10,4) ..
                   (10, 2) --
                   (10, 0);

\draw [ultra thick, rounded corners=6pt] (0,0) --
                   (0,2) .. controls (0,4) and (3,6) .. 
                   (5, 6) .. controls (7,6) and (10,4) ..
                   (10, 2) node [right] {$b_i$} --
                   (10, 0);

\draw [thick] (0, 0) -- (10, 0) node [below left] {$a_i$};
\draw (7,0) node [below left, circle, fill=white, inner sep=0] {$a_{i+1}$};
\draw [ultra thick] (0, 0) -- (10, 0);

\end{tikzpicture}
\end{center}
\caption{Nested bicorns.} \label{pic:nested bicorns}
\end{figure}

\noindent For a nested bicorn sequence, each pair of adjacent bicorns
$c_i$ and $c_{i+1}$ may be made disjoint after a small isotopy.

Let $D$ and $E$ be essential embedded discs in minimal position in a
compression body.  Then $a = \partial D$ and $b = \partial E$ is a
pair of essential simple closed curves in minimal position.  We say a
disc $F$ contained in $D \cup E$ is a \emph{bicorn disc} if the
boundary of the disc $F$ is a bicorn in $a \cup b$.

We omit the proof of the following observation.

\begin{proposition}
Let $D$ and $E$ be essential embedded discs in minimal position in a
compression body, and let $F_i$ be a bicorn disc in $D \cup E$ with
boundary $a_i \cup b_i$.  Then there is an arc $\gamma_i$ in
$D \cap E$ such that $F$ is the union of a subdisc $D_i$ of $D$
bounded by $a_i \cup \gamma_i$ and a subdisc $E_i$ of $E$ bounded by
$b_i$ and $\gamma_i$, which only intersect along $\gamma_i$.
\end{proposition}

A \emph{nested bicorn disc sequence} is a nested bicorn sequence in
which every bicorn $c_i = (a_i, b_i)$ bounds a disc $F_i = (D_i,
E_i)$, and furthermore $D_{i+1} \subset D_i$ for all $i$.

Two essential simple closed curves $a$ and $b$ in minimal position in
$S$, and a subinterval $a_i \subset a$ with endpoints in $a \cap b$,
determines a pre-train track $\tau_i'$ with a single switch, as
follows: discard $a \setminus a_i$, collapse $a_i$ to a point, and
smooth the tangent vectors as illustrated in Figure
\ref{pic:collapse}.  The dashed line labelled $a$ on the right hand
side of Figure \ref{pic:collapse} is not part of the pre-train track
$\tau_i'$.  Rather, it is drawn for comparison with the left hand
diagram.

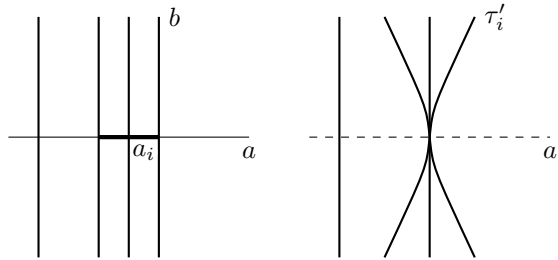
\begin{figure}[H]
\begin{center}
\begin{tikzpicture}[scale=0.4]

\draw [thick] (-3, -4) -- (-3, 4);
\draw [thick] (-1, -4) -- (-1, 4);
\draw [thick] (0, -4) -- (0, 4);
\draw [thick] (1, -4) -- (1, 4) node [right] {$b$};

\begin{scope}[xshift=+10cm]
    \draw [dashed] (-4, 0) -- (4, 0) node [below] {$a$};

    \draw [thick] (-3, -4) -- (-3, 4);
    \draw [thick] (0, -4) -- (0, 4);
    \draw [thick] (0, 0) .. controls (0.1, 1) .. (1.5, 4) node [right] {$\tau_i'$};
    \draw [thick] (0, 0) .. controls (-0.1, 1) .. (-1.5, 4);
    \draw [thick] (0, 0) .. controls (0.1, -1) .. (1.5, -4);
    \draw [thick] (0, 0) .. controls (-0.1, -1) .. (-1.5, -4);
\end{scope}

\draw (-4, 0) -- (4, 0) node [below] {$a$};
\draw (0.5, -0.5) node [circle, fill=white, inner sep=0] {$a_i$};
\draw [ultra thick] (-1, 0) -- (1, 0);

\end{tikzpicture}
\end{center} 
\caption{Smoothing intersections of $a$ and $b$ by collapsing $a_i$.} \label{pic:collapse}
\end{figure}

A pair of simple closed curves $a$ and $b$ in minimal position divide
the surface $S$ into a number of complementary regions, whose
boundaries consist of alternating innermost subarcs of $a$ and $b$. We
say a complementary region is a \emph{rectangle} if it is a disc,
whose boundary consists of exactly four innermost subarcs.

Given a pre-train tack $\tau$ contained in a surface $S$, a
\emph{bigon collapse} is a homotopy of the surface, supported in a
neighbourhood of a bigon, which maps the bigon to a single arc, as
illustrated in Figure \ref{pic:bigon collapse}.

\begin{figure}[H]
\begin{center}
\begin{tikzpicture}[scale=0.4]

\draw [thick] (0, 0) --
              (1, 0) .. controls (2, 0) and (2, 1) ..
              (3, 1) .. controls (4, 1) and (4, 0) ..
              (5, 0) --
              (6, 0);
              
\begin{scope}[yscale=-1]
\draw [thick] (0, 0) --
              (1, 0) .. controls (2, 0) and (2, 1) ..
              (3, 1) .. controls (4, 1) and (4, 0) ..
              (5, 0) --
              (6, 0);
\end{scope}

\draw [thick] (9,0) --(15, 0);

\end{tikzpicture}
\end{center} 
\caption{Collapsing a bigon.} \label{pic:bigon collapse}
\end{figure}
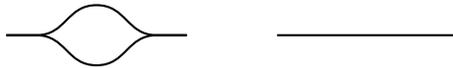

\noindent We break the proof of Proposition \ref{prop:splitting seq}
into the following three propositions.

We say a pre-track \emph{collapses} to a train track if there is a
sequence of bigon collapses which produces a train track.  We say a
bicorn $(a_i, b_i)$ is \emph{non-degenerate} if there is a
non-rectangular component of $S \setminus (a \cup b)$ whose boundary
intersects $a \setminus a_i$.  Non-degeneracy allows us to collapse a
pre-track to a track, as follows.

\begin{proposition}\label{prop:collapse}
Let $a$ and $b$ be simple closed curves in minimal position, and let
$(a_i, b_i)$ be a non-degenerate bicorn.  Then the pre-track $\tau'_i$
determined by $a_i$ and $b$ bigon collapses to a train track $\tau_i$.
Furthermore, $\tau_i$ is switch dual to the bicorn $c_i$ determined by
$(a_i, b_i)$.
\end{proposition}

A bicorn sequence $(a_i, b_i)$ is \emph{non-degenerate} if every
bicorn $(a_i, b_i)$ is non-degenerate.  If the initial bicorn
$(a_1, b_1)$ is non-degenerate, then this implies that every
subsequent bicorn is non-degenerate.

\begin{proposition}\label{prop:nested}
Let $a$ and $b$ be simple closed curves in minimal position, and let
$(a_i, b_i)$ be a collapsible nested bicorn sequence, and let $\tau_i$
be the corresponding train tracks.  Then $\tau_i$ is a carrying
sequence of train tracks: that is, $\tau_{i+1} \carried \tau_{i}$ for
all $i$.
\end{proposition}

\begin{proposition}\label{prop:discs}
Let $a$ and $b$ be simple closed curves which bound discs in a
compression body $V$.  Then there is a collapsible nested bicorn disc
sequence $\{ F_i \}_{i = 1}^n$ with bicorn boundaries
$\{ c_i \}_{i=1}^n$, as follows.  If we define $c_0 = a$, then the
sequence of simple closed curves $\{ c_i \}_{i=0}^n$ is a disc surgery
sequence connecting $a$ and $b$.
\end{proposition}

We now define a \emph{rectangular tie neighbourhood} for a train track
$\tau$ with a single switch.  This is a regular neighbourhood of
$\tau$ in the surface $S$ which is foliated by intervals transverse to
$\tau$, with a decomposition as a union of foliated rectangles with
disjoint interiors, with the following properties.  There is a single
rectangle containing the switch, which we shall call the \emph{switch
  rectangle}.  There is one rectangle for each branch, which we shall
call a \emph{branch rectangle}.  The branch rectangles have disjoint
closures.

\begin{figure}[H]
\begin{center}
\begin{tikzpicture}[scale=0.1]

\draw (0, -15) node [below] {The switch rectangle in the center is shaded.};

\filldraw [draw=black,fill=lightgray] (-15,-15) rectangle (15, 15);
\foreach \x in {-15,-13,...,15} { \draw (\x, -15) -- (\x, 15); }; 

\draw [thick] (0, 0) .. controls  (5, 0) and (10, 10) .. (15,
10);
\draw [thick] (0, 0) .. controls  (5, 0) and (10, -10) .. (15,
-10);

\draw [thick] (0, 0) -- (-15, 0);
\draw [thick] (0, 0) .. controls  (-5, 0) and (-10, 12) .. (-15,
12);
\draw [thick] (0, 0) .. controls  (-5, 0) and (-10, -12) .. (-15,
-12);

\draw (15, 5) rectangle (31, 15);
\foreach \x in {15,17,...,31} { \draw (\x, 5) -- (\x, 15); };
\draw [thick] (15, 10) -- (31, 10);

\begin{scope}[yshift=-20cm]
\draw (15, 15) rectangle (31, 5);
\foreach \x in {15,17,...,31} { \draw (\x, 5) -- (\x, 15); }; 
\draw [thick] (15, 10) -- (31, 10);
\end{scope}

\draw (-15, 3) rectangle (-31, -3);
\foreach \x in {-15,-17,...,-31} { \draw (\x, 3) -- (\x, -3); };
\draw [thick] (-15, 0) -- (-31, 0);

\begin{scope}[yshift=-12cm]
\draw (-15, 3) rectangle (-31, -3);
\foreach \x in {-15,-17,...,-31} { \draw (\x, 3) -- (\x, -3); };
\draw [thick] (-15, 0) -- (-31, 0);
\end{scope}

\begin{scope}[yshift=12cm]
\draw (-15, 3) rectangle (-31, -3);
\foreach \x in {-15,-17,...,-31} { \draw (\x, 3) -- (\x, -3); };
\draw [thick] (-15, 0) -- (-31, 0);
\end{scope}
    
\end{tikzpicture}
\end{center}
\caption{A rectangular tie neighbourhood for a train track.} \label{fig:rtie}
\end{figure}

We define a \emph{rectangular tie neighbourhood} for a pre-track with
a single switch as above, except that we allow a single rectangle to
contain multiple parallel branches.  We observe that parallel branches
in a single rectangle make up the boundaries of bigons in the
pre-track, and these may be collapsed by a homotopy supported in the
union of the switch rectangle and the rectangle containing the
branches.  In particular, if all bigon complementary regions are
contained in the rectangular tie neighbourhood, then collapsing all
bigons produces a train track, for which the rectangular tie
neighbourhood of the pre-track is a rectangular tie neighbourhood for
the train track.

We now define a collection of foliated rectangles determined by $a$
and $b$.  All of the rectangular tie neighbourhoods we construct will
be subcollections of these rectangles.  Isotope $a$ and $b$ into
minimal position.  For each intersection point $x$ in $a \cap b$ we
take a rectangular neighbourhood $R_x$, which we shall call a
\emph{vertex rectangle}.  The rectangle has four \emph{corners}, one
in each of the quadrants formed by the local intersection of $a$ and
$b$, and alternating sides parallel to $a$ and $b$. We foliate $R_x$
by arcs parallel to $a \cap R_x$, so that the two sides of the
rectangle parallel to $a$ are leaves of the foliation.  This is
illustrated below in Figure \ref{fig:rx}.

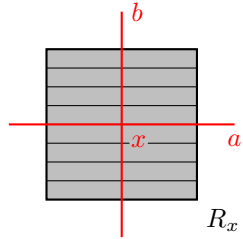
\begin{figure}[H]
\begin{center}
\caption{A rectangle $R_x$ determined by a point $x \in a \cap
  b$.} \label{fig:rx}
\begin{tikzpicture}[scale=0.5]

\filldraw [thick,color=black,fill=lightgray] (-2,-2) rectangle (2, 2);
\foreach \y in {-2,-1.5,...,2} { \draw (-2, \y) -- (2, \y); };

\draw [thick, red] (-3, 0) -- (3, 0) node [below] {$a$};

\draw [thick, red] (0, -3) -- (0, 3) node [right] {$b$};

\draw (0.45, -0.45) node [red, circle, fill=lightgray, inner sep=0] {$x$};
\draw (2,-2) node [below right] {$R_x$};

\end{tikzpicture}
\end{center}
\end{figure}

Now suppose that $\alpha$ is a component of $a \setminus b$, with
endpoints $x$ and $y$.  We define a rectangle $R_\alpha$, called an
\emph{$a$-rectangle}, as follows.  It is a rectangle in $S$,
containing $\alpha \setminus (R_x \cup R_y)$, two of whose sides
consist of the sides of $R_x$ and $R_y$ which are parallel to $b$ and
intersect $\alpha$.  The other two sides consist of properly embedded
arcs parallel to $\alpha \setminus (R_x \cup R_y)$.  We shall foliate
this rectangle with arcs parallel to $a \cap R_\alpha$ such that the
two sides parallel to $\alpha$ are leaves of the foliation.  This is
illustrated in Figure \ref{pic:Ralpha}.

\begin{figure}[H]
\begin{center}
\begin{tikzpicture}[scale=0.4]

\draw [thick,black] (-2,-2) rectangle (2, 2);
\foreach \y in {-2,-1,...,2} { \draw (-2, \y) -- (2, \y); };

\begin{scope}[xshift=16cm]
\draw [thick,black] (-2,-2) rectangle (2, 2);
\foreach \y in {-2,-1,...,2} { \draw (-2, \y) -- (2, \y); };
\end{scope}

\draw [thick, color=black, fill=lightgray] (2,-2) rectangle (14, 2);
\foreach \y in {-2,-1,...,2} { \draw (2, \y) -- (14, \y); };

\draw [thick, red] (-4, 0) -- (20, 0) node [below] {$a$};

\draw [thick, red] (0, -4) -- (0, 4) node [right] {$b$};
\draw [thick, red] (16, -4) -- (16, 4) node [right] {$b$};

\draw (0,0) node [red, below right] {$x$};
\draw (16,0) node [red, below right] {$y$};
\draw (8, 0) node [red, below] {$\alpha$};

\draw (8, -2) node [below] {$R_\alpha$};

\draw (-2, -2) node [below] {$R_x$};
\draw (18, -2) node [below] {$R_y$};

\end{tikzpicture}
\end{center}
\caption{Rectangles determined by a component $\alpha$ of
  $a \setminus b$.} \label{pic:Ralpha}
\end{figure}
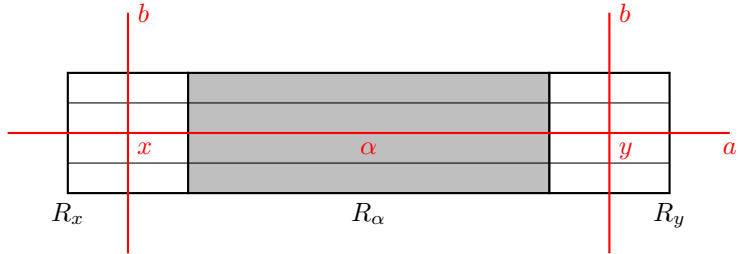

We do the same for components $\beta$ of $b \setminus a$, but this
time foliated by arcs crossing $\beta$ exactly once, as illustrated in
Figure \ref{pic:Rbeta}.  We shall call these rectangles
\emph{$b$-rectangles}.

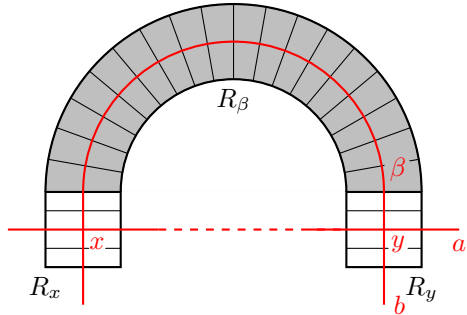
\begin{figure}[H]
\begin{center}
\begin{tikzpicture}[scale=0.25]

\draw [fill=lightgray] ([shift=(0:10cm)]8,2) arc (0:180:10cm);
\draw [fill=white] ([shift=(0:6cm)]8,2) arc (0:180:6cm);

\draw [thick,black] (-2,-2) rectangle (2, 2);
\foreach \y in {-2,-1,...,2} { \draw (-2, \y) -- (2, \y); };

\begin{scope}[xshift=16cm]
\draw [thick,black] (-2,-2) rectangle (2, 2);
\foreach \y in {-2,-1,...,2} { \draw (-2, \y) -- (2, \y); };
\end{scope}

\foreach \x in {0,10,...,180} { \draw (8,2) ++(\x:6cm) -- ++(\x:4cm); };

\draw [thick] ([shift=(0:6cm)]8,2) arc (0:180:6cm);

\draw [thick] ([shift=(0:10cm)]8,2) arc (0:180:10cm);

\draw [thick, red] ([shift=(0:8cm)]8,2) arc (0:180:8cm);

\draw [thick, red] (-4, 0) -- (4, 0);
\draw [thick, red, dashed] (4, 0) -- (14, 0);
\draw [thick, red] (12, 0) -- (20, 0) node [below] {$a$};

\draw [thick, red] (0, -4) -- (0, 2);
\draw [thick, red] (16, -4) node [right] {$b$} -- (16, 2);
\draw (16.75, 3.2) node [red, circle, fill=lightgray, inner sep=0] {$\beta$};

\draw (0.75, -0.75) node [red, circle, fill=white, inner sep=0] {$x$};
\draw (16.75, -0.75) node [red, circle, fill=white, inner sep=0] {$y$};

\draw (-2, -2) node [below] {$R_x$};
\draw (18, -2) node [below] {$R_y$};

\draw (8, 8) node [below] {$R_\beta$};

\end{tikzpicture}
\end{center}
\caption{Rectangles determined by a component $\beta$ of $b
  \setminus a$.} \label{pic:Rbeta}
\end{figure}

Finally, suppose that $f$ is a rectangle component of
$S \setminus (a \cup b)$.  The \emph{face rectangle} $R_f$ is a
foliated rectangle lying inside $f$, whose sides consist of the four
sides of the $a$- and $b$-rectangles which meet $f$.  The foliation
consists of arcs parallel to $a$, such that the two $a$-sides of the
rectangle are leaves of the foliation.  This is illustrated below in
Figure \ref{pic:Rf}.

\begin{figure}[H]
\begin{center}
\begin{tikzpicture}[scale=0.3]

\draw [thick, color=black, fill=lightgray] (2,2) rectangle (14, 8);

\draw [thick,black] (-2,-2) rectangle (2, 2);
\foreach \y in {-2,-1,...,2} { \draw (-2, \y) -- (2, \y); };

\begin{scope}[xshift=16cm]
\draw [thick,black] (-2,-2) rectangle (2, 2);
\foreach \y in {-2,-1,...,2} { \draw (-2, \y) -- (2, \y); };
\end{scope}

\foreach \y in {-2,-1,...,2} { \draw (2, \y) -- (14, \y); };

\draw [thick, red] (-4, 0) -- (20, 0) node [below] {$a$};

\draw (0.75,-1) node [red, circle, fill=white, inner sep=0] {$w$};
\draw (16.75,-1) node [red, circle, fill=white, inner sep=0] {$z$};
\draw (8, -1) node [red, circle, fill=white, inner sep=0] {$\alpha'$};

\draw (8, -2) node [below] {$R_{\alpha'}$};

\draw (-2, 5) node [left] {$R_\beta$};
\draw (18, 5) node [right] {$R_{\beta'}$};

\draw (-2, -2) node [below] {$R_w$};
\draw (18, -2) node [below] {$R_z$};

\draw [thick, red] (0, 4) -- (0, 6);
\draw [thick, red] (16, 4) -- (16, 6);

\draw [thick, black] (-2, 2) -- (-2, 8);
\draw [thick, black] (2, 2) -- (2, 8);

\draw [thick, black] (14, 2) -- (14, 8);
\draw [thick, black] (18, 2) -- (18, 8);

\foreach \y in {3,4,...,7} { \draw (-2, \y) -- (18, \y); };

\draw [thick, black] (2, -2) -- (14, -2);
\draw [thick, black] (2, 12) -- (14, 12);

\draw [thick, red] (0, -4) -- (0, 4) node [right, circle, fill=white,
inner sep=0] {$\beta$};
\draw [thick, red] (16, -4) -- (16, 4) node [right, circle, fill=white,
inner sep=0] {$\beta'$};

\begin{scope}[yshift=+10cm]
\draw [thick,black] (-2,-2) rectangle (2, 2);
\foreach \y in {-2,-1,...,2} { \draw (-2, \y) -- (2, \y); };

\begin{scope}[xshift=16cm]
\draw [thick,black] (-2,-2) rectangle (2, 2);
\foreach \y in {-2,-1,...,2} { \draw (-2, \y) -- (2, \y); };
\end{scope}

\foreach \y in {-2,-1,...,2} { \draw (2, \y) -- (14, \y); };

\draw [thick, red] (-4, 0) -- (20, 0) node [below] {$a$};

\draw [thick, red] (0, -4) -- (0, 4) node [right] {$b$};
\draw [thick, red] (16, -4) -- (16, 4) node [right] {$b$};

\draw (0.75,-1) node [red, circle, fill=white, inner sep=0] {$x$};
\draw (16.75,-1) node [red, circle, fill=white, inner sep=0] {$y$};
\draw (8, -1) node [red, circle, fill=white, inner sep=0] {$\alpha$};
\draw (8, 2) node [above] {$R_\alpha$};

\draw (-2, 2) node [above] {$R_x$};
\draw (18, 2) node [above] {$R_y$};
\end{scope}

\draw (8, 5) node [circle, fill=lightgray, inner sep=0] {$R_f$};

\end{tikzpicture}
\end{center}
\caption{A face rectangle determined by a rectangular face $f$ of
  $S \setminus ( a \cup b )$.} \label{pic:Rf}
\end{figure}
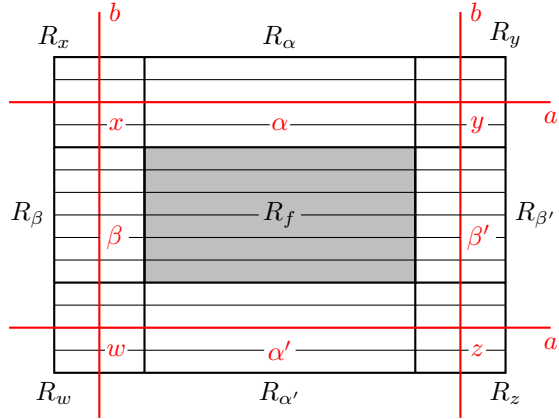

We shall denote the resulting foliation of this subset of $S$ by
$\calF$. This foliates all of $S$ except for the (slightly shrunken)
non-rectangular regions of $S \setminus (a \cup b)$.

We now prove Proposition \ref{prop:collapse}.  We will construct a
foliated region in $S$ which is a union of rectangles, and show that
it is a tie neighbourhood for a train track with a single switch.

\begin{proof}[Proof (of Proposition \ref{prop:collapse}).]
The foliated region $\calF$ is a union of foliated rectangles.  We
will build a rectangular tie neighbourhood $\calF_i$ for the pre-track
$\tau'_i$ which will consist of a subcollection of these rectangles,
and which may contain branch rectangles with multiple edges.  There
will be no complementary regions of $\calF_i$ which are rectangles, so
the pre-track $\tau'_i$ will collapse to a train track $\tau_i$.

The foliated region $\calF_i$ consists of all vertex rectangles and
all $b$-rectangles, together with all the face rectangles which are
contained in rectangular components of $S \setminus (a_i \cup b)$, as
well as all $a$-rectangles adjacent to an included face rectangle.

Let $a^+_i$ be the maximal subarc of $a$ with endpoints in $a \cap b$
contained in the connected component of $\calF_i \cap a$ containing
$a_i$.  In particular $a_i \subseteq a^+_i$.  As the bicorn
$(a_i, b_i)$ is non-degenerate, there is at least one non-rectangle
region of $S \setminus (a \cup b)$ with a boundary edge in
$a \setminus a_i$.  Any non-rectangle region of
$S \setminus (a \cup b)$ is contained in a non-rectangle region of
$S \setminus (a_i \cup b)$, so $a^+_i$ does not consist of all of
$a_i$, and therefore is an interval.

The union of the vertex rectangles and the $a$-rectangles in $\calF_i$
meeting $a^+_i$ is a regular neighbourhood of $a^+_i$, and is a
foliated rectangle containing the unique switch of the pre-track
$\tau_i$.  We shall denote this rectangle by $R^+_i$.  This will be
the switch rectangle for the rectangular tie neighbourhood.

We must now show that all components of $\calF_i \setminus R^+_i$ are
foliated rectangles (perhaps containing multiple branches of $\tau_i$)
and that no components of $S \setminus \calF_i$ is a rectangle.


If a component of $\calF_i \setminus R^+_i$ has no face rectangles,
then it is a union of $b$-edge rectangles and vertex rectangles.  Let
$\calB$ be the union of the vertex rectangles and the $b$-rectangles;
so $\calB$ is a regular neighbourhood of $b$.  The intersection
$\calB \cap R^+_i$ consists of a subset of the vertex rectangles.
Thus $\calB \setminus R^+_i$ is a union of rectangles with disjoint
closures.  We will refer to the components of $\calB \setminus R^+_i$
as \emph{$b$-strips}.  Each $b$-strip is a rectangle whose boundary
consists of four edges: two parallel to $a$, and contained in
$\partial R^+_i$, and the other two parallel to $b$, and parallel to
properly embedded arcs in the complement of $R^+_i$.  We say two
$b$-strips are \emph{parallel} if their $b$-parallel edges cobound a
rectangle in $S \setminus R^+_i$.


We say two face rectangles are \emph{vertically adjacent} if they
border a common $a$-rectangle not in $R^+_i$.  We say a maximal
collection of vertically adjacent rectangles, together with the
$a$-rectangles between them, is a \emph{face strip}.  A face strip is
a rectangle whose boundary consists of four edges, two parallel to $a$
and contained in $\partial R^+_i$, and two parallel to $b$, and are
properly embedded parallel arcs in the complement of $R^+_i$.  We say
two face strips are \emph{parallel} if the $b$-parallel edges are
parallel arcs in the complement of $R^+_i$.

We say a face strip is \emph{parallel} to a $b$-strip if their
$b$-parallel edges cobound a rectangle in $S \setminus R^+_i$.  We
observe that a face strip is adjacent to two $b$-strips, one on each
side, and all three strips are parallel to each other.

A component of $\calF_i \setminus R^+_i$ which contains a face
rectangle is a union of face strips and $b$-strips.  Each pair of face
strips is separated by a $b$-strip, and the observation in the
paragraph above ensures that all of the strips are parallel.

We now verify that all of the components of $\calF_i \setminus R^+_i$
are disjoint.  Any two components consisting only of $b$-strips are
disjoint.  Let $R_1$ and $R_2$ be a pair of components of
$\calF_i \setminus R^+_i$ which contain a corner in common.  Each
corner lies in the boundary of a face rectangle, a $b$-rectangle, an
$a$-rectangle and a vertex rectangle.  Suppose the face rectangle lies
in one of the components, say $R_1$.  Then the $b$-rectangle also lies
in $R_1$.  If the vertex and $a$-rectangles do not lie in $R_1$, then
they lie in $R^+_i$, so there can be no rectangle $R_2$ intersecting
$R_1$, a contradiction.  So neither $R_1$ nor $R_2$ contain the face
rectangle.  If the $a$-rectangle lies in $R_1$, then so must the face
rectangle beside it, so the $a$-rectangle is also not contained in
either $R_1$ or $R_2$.  If $R_1$ contains the $b$-rectangle, then it
either contains the vertex rectangle, or the vertex rectangle lies in
$a^+_i$.  In either case, none of the other rectangles can lie in
$R_2$, a contradiction.
\end{proof}

Proposition \ref{prop:nested} follows from work of Penner and Harer
\cite{ph}.

\begin{proof}[Proof (of Proposition \ref{prop:nested}).]
Any closed train route in $\tau_{i+1}$ is a union of arcs of $b$
starting and ending at $a_{i+1}$.  As $a_{i+1} \subset a_i$, it is
also a train route in $\tau_i$.  The train track $\tau_{i+1}$ may then
be obtained from $\tau_i$ by splitting and shifting, by Penner and
Harer \cite{ph}*{Theorem 2.4.1}.
\end{proof}

Before we prove Proposition \ref{prop:discs} we need the following
observation. 


\begin{proposition} \label{prop:bicorn essential} %
Suppose that $a$ and $b$ are essential simple closed curves in minimal
position.  Then any bicorn contained in $a \cup b$ is an essential
simple closed curve.
\end{proposition}


Proposition \ref{prop:bicorn essential} is standard and we omit the
proof.


\begin{proof}[Proof (of Proposition \ref{prop:discs}).]
Let $a = \partial D$ and $b = \partial E$.  We first show how to
choose an initial non-degenerate bicorn $(a_1, b_1)$.  By an Euler
characteristic argument, there must be at least one complementary
region of $a \cup b$ which is not a rectangle.

Let $\alpha$ be a component of $a \setminus b$ which lies in the
boundary of one of the non-rectangular regions of the complement of
$a \cup b$.  Let $b_1 \subset b$ be a choice of returning arc
determined by an outermost disc $E_1$ in $E$, and let $a_1$ be the
subinterval of $a$:
\begin{itemize}
\item with the same endpoints as $b_1$,
\item which is disjoint from $\alpha$.
\end{itemize}
The bicorn $c_1$ determined by $(a_1, b_1)$ bounds a
disc corresponding to disc surgery of $D$ along $E_1$, and we shall
denote this disc by $F_1$.  This disc $F_1$ is essential by
Proposition \ref{prop:bicorn essential}.

Now suppose we have constructed the $i$-th nested bicorn disc
$F_i = (D_i, E_i)$ with bicorn boundary $c_i = (a_i, b_i)$.  The
intersection $E \cap F_i$ is equal to $E_i \cup (E \cap D_i)$.  Choose
$E_{i+1}$ to be an outermost disc of $E$ with respect to $E \cap D_i$,
which is not $E_i$.  Then $\gamma_{i+1} = E_{i+1} \cap D_i$ bounds a
disc in $D_i$ which we shall choose to be $D_{i+1}$.  We shall set
$F_{i+1} = D_{i+1} \cup E_{i+1}$, which has bicorn boundary
$c_{i+1} = (a_{i+1}, b_{i+1})$, where $a_{i+1} \subset a_i$ to be
$D_{i+1} \cap a$ and $b_{i+1}$ to be $E_{i+1} \cap b$.  The disc
$F_{i+1}$ is essential by Proposition \ref{prop:bicorn essential}, and
furthermore, is obtained from disc surgery of $F_i$ along $E_{i+1}$.
We have therefore constructed the next disc in the nested bicorn disc
sequence, as required.
%
\end{proof}



\vskip 20pt

\noindent Joseph Maher \\
CUNY College of Staten Island and CUNY Graduate Center \\
\url{joseph.maher@csi.cuny.edu} \\

\noindent Saul Schleimer \\
University of Warwick \\
\url{s.schleimer@warwick.ac.uk }


\end{document}